\documentclass{amsart}

\newtheorem{theorem}{Theorem}
\newtheorem{lemma}{Lemma}
\newtheorem{proposition}{Proposition}
\theoremstyle{definition}
\newtheorem{definition}{Definition}
\theoremstyle{remark}

\newtheorem{corollary}{Corollary}
\numberwithin{equation}{section}

\usepackage{amssymb}
\usepackage{empheq}
\usepackage{stmaryrd}
\usepackage{enumerate}
\usepackage{multicol}
\usepackage[shortlabels]{enumitem}
\usepackage{calc}
\usepackage{tikz}
\usepackage{listings}

\lstdefinestyle{derivation} {
    breakatwhitespace=false,         
    breaklines=true,                 
    captionpos=b,                    
    keepspaces=true,                 
    numbers=left,                    
    numbersep=5pt,
    escapechar=\&
}

\newcommand*{\Comment}[1]{\hfill\makebox[9.0cm][l]{#1}}

\usetikzlibrary{automata,positioning}
\usetikzlibrary{arrows}

\usepackage[left=2.0cm,%
                right=2.0cm,%
                top=2.5cm,%
                bottom=2.5cm,%
                headheight=12pt,%
                a4paper]{geometry}%

\newenvironment{claim}[1]{\par\noindent\underline{Claim:}\space#1}{}
\newenvironment{claimproof}[1]{\par\noindent\underline{Proof:}\space#1}{\hfill $\blacksquare$\newline}

\begin{document}

\title{Dynamic extensions for the logic of knowing why with public announcements of formulas}

\author{Nicholas Pischke}
\address{Hoch-Weiseler Str. 46, Butzbach, 35510, Hesse, Germany}
\email{pischkenicholas@gmail.com}


\begin{abstract}
In this paper, we address the \emph{logic of knowing why}, an example of a non-standard epistemic logic dealing with justified knowledge via a new epistemic operator, under the extensions with ideas from dynamic epistemic logic, namely public announcements. Through the additional notions present in the knowing why context, we consider two possible variants, namely the extensions by \emph{(i)}: public announcements of a formula and by \emph{(ii)}: public announcements of reasons, although the deeper analysis of the latter is left for future work. We consider another logical operator, the conditional knowing-why operator, for which we study the applications to the axiomatization of public announcements as well as the solely framework. At the end, we consider the logical expressivity of these different logics in comparison to each other, and thus we show one of the main problems with the usual process of proving completeness through translation in the context of logics with public announcements.
\end{abstract}

\maketitle

\section{Introduction}
The logic of knowing why, as introduced by Xu, Wang and Studer in \cite{XWS2016}, considers a synthesis of ideas very similar to justification logic together with the classical notions of epistemic logic to provide a framework for reasoning not only about knowing formulas but also knowing \emph{why} in the concrete sense of knowing explanations for formulas in all concerned worlds of an agent. For this, the authors introduced a new modal operator $Ky_a$ in extension to the basic epistemic framework. In contrast to justification logics however, the knowing-why operator is inherently unexplicit about the actual reasons in concerned situations.\\

The logic of knowing why is located in the framework of ideas for extensions of basic epistemic logic by more graded notions knowing, e.g. knowing who, knowing how, etc., so called non-standard epistemic logics, originating in Hintikka's serminal work \cite{Hin1962} and its sequels and again emerging as a subject of other recent contributions to epistemic logics, see e.g. \cite{Wan2018}. This synthesis of knowing and knowing why can also be found in other, conceptually different realizations, like in the context of justification logics in combination with standard epistemic logic as e.g. in \cite{AN2005}. Such a framework, capable of analyzing origins of agents knowledge deeper, promises interesting applications, especially after an enhancement by dynamic notions, in a same way as for these related systems, like in classical dynamic epistemic logic and e.g. for justification logic by Renne in e.g. \cite{Ren2008}, \cite{Ren2012} and \cite{Ren2011}. Although dynamic epistemic logic encompasses many different notions and systems, we focus on the concept of public announcements for the extension. While originally introduced by Plaza in his seminal work \cite{Pla1989}, we will manly adapt, and refer to, the presentation in the monograph \cite{DHK2007}.\\

As there are two different kinds of notions concerned, namely solely knowing and knowing why(or knowing reasons) we also can imagine two notions of public announcements from this context. Either publicly announcing \emph{that} a formula has to hold or even publicly announcing \emph{why} a formula has to hold. We will only examine the first concept in this paper, i.e. an enhancement of the logic of knowing why with a classical public announcement operator for formulas is considered. The second one could then follow the ideas and concepts presented by Eijck, Gattinger and Wang in \cite{EGW2016}. We then find that, in difference to usual ehancements by public announcement operators, the logic of knowing why bears some deeper semantical idiosyncrasies in cooperation with these notions, making a classical, that is in the context of public announcements a reduction style axiomatization impossible as they add expressive power. We then introduce a relativized version of the knowing-why operator, following suggestions made in \cite{XWS2016}, to provide a possible workaround for these problems. As a main goal, we provide an axiomatization of this latter logic which is then proved sound and complete with respect to the basic model classes presented in the original paper.\\

We then provide expressivity comparisons between the logic incorporating public announcements and the logic using the relativized knowing why operator. Together with this, we again address the issues with axiomatizing these different logics in a reduction sense.
\section{Preliminaries}
The main purpose of this section is to provide an overview of the work done about the \emph{logic of knowing why}, by Xu et. al. in \cite{XWS2016}, to create a common ground on which the dynamic extensions and other modifications later take place. For the following, it's assumed that the reader is familiar with the basic notions of propositional classical modal logic, Kripke-frames and the concepts of basic epistemic logic.
\subsection{The logic of knowing why}
The underlying language for the logic of knowing why, in the following denoted by \textbf{ELKy} in correspondence to the initial paper, is defined with the \emph{BNF}
\begin{equation*}
\mathcal{L}_{ELKy}:\phi ::= p\; |\; \neg\phi\; |\; (\phi\land\phi )\; |\; K_a\phi\; |\; Ky_a\phi
\end{equation*}
with $p\in\mathcal{P}$ and $a\in\mathcal{A}$. The sets $\mathcal{P}$ and $\mathcal{A}$ used here are respectively the sets for the atomic propositions, being countably infinite, and the countable set of agents. All other common connectives like $\rightarrow$, $\leftrightarrow$ as well as $\bot$ and $\top$ are defined in the same way as in classical propositional logic. Following the general notations of epistemic logic, the dual of the modal operator $K_a$ is defined as $\widehat{K}_a\phi =\neg K_a\neg\phi$. As it follows naturally, while $K_a$ reads like \emph{the agent $a$ knows ... is the case}, the new modal operator $Ky_a$ read like \emph{the agent $a$ knows why ... is the case}.\\

The semantics of the logic \textbf{ELKy} was defined in a classical modal model-theoretic sense with following the approach of Fitting models for justification logic. These originated in \cite{Fit2005} as a natural extension to the approach of a possible-world Kripke model in the context of classical modal logic, enhanced by additional functions to model local, that is world-specific, relationships between explanations and formulas. Followingly, an $\mathbf{ELKy}$-model $\mathfrak{M}=\langle W,E,\{ R_a\,|\, a\in\mathcal{A}\},\mathcal{E},V\rangle$ is defined over \emph{(i)} a non-empty set of worlds $W$, also called the domain $\mathcal{D}(\mathfrak{M})$ of $\mathfrak{M}$, \emph{(ii)} a non-empty set of explanations $E$, \emph{(iii)} an own accessibility relation $R_a\subseteq W\times W$ for each agent $a\in\mathcal{A}$, \emph{(iv)} an admissible explanation function $\mathcal{E}: E\times\mathcal{L}_{ELKy}\to 2^W$\footnote{Note that, in general for a set $X$, we identify $2^X$ with the power set of $X$.} and $\emph{(v)}$ a basic evaluation function $V:\mathcal{P}\to 2^W$. A pointed version of a model, i.e. a combination of a model and a designated world is then denoted by $(\mathfrak{M},w)$.\\

As it was defined in the initial paper, the set $E$, holding the possible explanations for formulas, has to satisfy two conditions, namely \emph{(1)} holding a designated explanation $e$ and \emph{(2)} being closed under a explanation-combination operator $\cdot :E\times E\to E$, i.e. $s,t\in E$ implying $(s\cdot t)\in E$. With $E$ being independent of an agents view, it can be seen as an omnipresent domain of explanations.\\

The explanation function $\mathcal{E}$, relating the worlds $w\in W$ to a formula $\phi$ and an explanation $s$, in the sense of $s$ being an explanation of $\phi$ at some world $w$, also has to fulfill two conditions about its behavior, namely \emph{(1)} whenever a formula $\phi$ is in a designated set $\Lambda$, it holds that $\mathcal{E}(e,\phi )=W$ and \emph{(2)} the function $\mathcal{E}$ distributes over $\cdot$-application in combination with a \emph{modus ponens} style inference with $\mathcal{E}(s,\phi\rightarrow\psi )\cap\mathcal{E}(t,\phi)\subseteq\mathcal{E}((s\cdot t),\psi)$. This designated set $\Lambda$, introduced by the authors in the original paper, is called the \emph{tautology ground}, simply a set consisting of valid formulas, which represent a fixed argumentation ground for all agents which are regarded as self-evidently true.\footnote{Following from this, the special explanation $e$ is called the self-evident explanation.} This stands in similarity with constant specifications in justification logics, and in a concrete sense, it helps the agents to make more justified conclusions.\\

The accessibility relation $R_a$ for each agent $a$ is in the following required to be a so called $\mathcal{S}5$ relation, i.e. being \emph{(i)} reflexive, that is for all $w$, $(w,w)\in R_a$, \emph{(ii)} transitive, that is for all $w,u,v$, if $(w,u)\in R_a$ and $(u,v)\in R_a$, then $(w,v)\in R_a$ and \emph{(iii)} symmetric, that is for all $w,u$, if $(w,u)\in R_a$ then $(u,w)\in R_a$. Although in the realm of modal logics there are many other classes of frames and models of particular interest, defined e.g. over different restrictions of the accessibility relations, the main emphasis will be on those $\mathcal{S}5$ models.\footnote{Note, that the corresponding class of models is denoted by $\mathcal{K}y\mathcal{S}5$, while the basic class of models with no restrictions for the accessibility relations is simply denoted by $\mathcal{K}y$.}\\

Local satisfiability, that is the validity of a formula in a specific world $w$ of a model $\mathfrak{M}$, is then recursively defined over the relation $\models$ with
\begin{align*}
&(\mathfrak{M},w)\models p\text{ iff }w\in V(p) &&(\mathfrak{M},w)\models\neg\phi\text{ iff }(\mathfrak{M},w)\not\models\phi\\
&(\mathfrak{M},w)\models\phi\land\psi\text{ iff }(\mathfrak{M},w)\models\phi\text{ and }(\mathfrak{M},w)\models\psi &&(\mathfrak{M},w)\models K_a\phi\text{ iff }\forall v\in W:(w,v)\in R_a\text{ implies }(\mathfrak{M},v)\models\phi
\end{align*}
for the classical operators from epistemic logic, and with
\begin{align*}
&(\mathfrak{M},w)\models Ky_a\phi\text{ iff (1): }(\mathfrak{M},w)\models K_a\phi\text{ and (2): }\exists t\in E:\forall v\in W:(w,v)\in R_a\text{ implies that }v\in\mathcal{E}(t,\phi ) &&\quad\quad\;
\end{align*}
for the new operator $Ky_a$. The case of a formula being valid in all worlds $w$ of a model $\mathfrak{M}$ is simply denoted by $\mathfrak{M}\models\phi$.\\

The notions for \emph{local} semantic deduction of a formula $\phi$ from a set of formulas $\Gamma$ in a specific model class $\mathbb{M}$, $\Gamma\models_\mathbb{M}\phi$, as well as the formal proof in a Hilbert-style axiomatic system $\mathbb{S}$, $\Gamma\vdash_\mathbb{S}\phi$, are defined as usual in the context of modal propositional logics.\\

In \cite{XWS2016}, the authors proposed an axiomatic system, the system $\mathbb{SKY}$ shown below, for which they've proved soundness and completeness with respect to \emph{local} semantic deduction in the class of all $\mathcal{S}5$-$\mathbf{ELKy}$-models defined as above.\footnote{The presentation of the theorems and corresponding proofs are omitted here(s. \emph{Theorem 13} and \emph{Theorem 23} in \cite{XWS2016}).}
\begin{definition}
The system $\mathbb{SKY}$ is given by the following axioms and rules:
\begin{description}[leftmargin=!,labelwidth=\widthof{($(NKy)$):}]
\item [$(PT)$] $\text{the classical propositional axioms}$
\item [$(K)$] $K_a(\phi\rightarrow\psi )\rightarrow (K_a\phi\rightarrow K_a\psi )$
\item [$(Ky)$] $Ky_a(\phi\rightarrow\psi )\rightarrow (Ky_a\phi\rightarrow Ky_a\psi )$
\item [$(T)$] $K_a\phi\rightarrow\phi$
\item [$(4)$] $K_a\phi\rightarrow K_aK_a\phi$
\item [$(5)$] $\neg K_a\phi\rightarrow K_a\neg K_a\phi$
\item [$(PS)$] $Ky_a\phi\rightarrow K_a\phi$
\item [$(4YK)$] $Ky_a\phi\rightarrow K_aKy_a\phi$
\item [$(MP)$] $\text{From }\phi\text{ and }\phi\rightarrow\psi\text{, infer }\psi$
\item [$(NK)$] $\text{From }\vdash\phi\text{, infer }\vdash K_a\phi$
\item [$(NKy)$] $\text{From }\phi\in\Lambda\text{, infer }\vdash Ky_a\phi$
\end{description}
\end{definition}
\section{Public announcements of formulas}
As the main topic of this paper, we will now consider the logic of knowing why under the extensions with the operator for public announcements of formulas. The basic notions of this new operator are defined according to the common notions of classical public announcement logic $\mathbf{PA}$. We will denote this new logic with $\mathbf{PAFKy}$ where $PAF$ emphasizes the announcement of formulas additionally.
\begin{definition}
For a countable set of agents $\mathcal{A}$, a countably infinite set of atomic propositions $\mathcal{P}$, the language of the logic $\mathbf{PAFKy}$ is defined with
\begin{equation*}
\mathcal{L}_{PAFKy}:\phi ::= p\; |\;\neg\phi\; |\; (\phi\land\phi )\; |\; K_a\phi\; |\; Ky_a\phi\; |\; [\phi ]\phi
\end{equation*}
where $p\in\mathcal{P}$ and $a\in\mathcal{A}$.
\end{definition}
Verbally, the construct $[\phi ]\psi$ translates to ``after the public announcement of the formula $\phi$, $\psi$  holds``. The semantics of this augmented logic are again defined over the same $\mathcal{S}5$-relational knowing-why models $\mathfrak{M}$, being constructed as presented in the preliminaries. We extend the satisfaction relation $\models$ for the handling of this new operator with
\begin{equation*}
(\mathfrak{M},w)\models [\phi]\psi\text{ iff }(\mathfrak{M},w)\models \phi\text{ implies }(\mathfrak{M}|\phi ,w)\models\psi
\end{equation*}
The construct $(\mathfrak{M}|\phi ,w)$ represents the pointed version of an \emph{updated} model, $\mathfrak{M}|\phi$, with the following construction.
\begin{definition}
The model $\mathfrak{M}$ after the public announcement of $\phi$, $\mathfrak{M}|\phi =\langle W',E,\{ R'_a\, |\,a\in\mathcal{A}\} ,\mathcal{E}', V'\rangle$, is defined with
\begin{align*}
&W' = \{ w\in W\, |\, (\mathfrak{M},w)\models\phi\}\\
&R'_a = R_a\cap (W'\times W')\text{ f.a. }a\in\mathcal{A}\\
&V'(p) = V(p)\cap W'\text{ f.a. }p\in\mathcal{P}\\
&\mathcal{E}'(t,\phi)=\mathcal{E}(t,\phi)\cap W'\text{ f.a. }t\in E,\phi\in\mathcal{L}_{PAFKy}
\end{align*}
\end{definition}
Note, that $E$ from the previous model is not changed in the ``update`` process, as we view the set of explanations as disconnected from the classical Kripkean part of the model, being omnipresent in a given larger context and only interacting with the possible-worlds part of the model via the function $\mathcal{E}$. For the right hand definition of $W'$, we will write $\llbracket\phi\rrbracket_{\mathfrak{M}}$ in the following. One can also imagine the operator $|$ for the update of a model by a publicly announced formula $\phi$ as the application of a function $| :\mathcal{K}y\mathcal{S}5\times\mathcal{L}_{PAFKy}\to\mathcal{K}y\mathcal{S}5$.\\\\
The corresponding dual to this operator $[\phi]$ is denoted by $\langle\phi\rangle$ and semantically defined by the following:
\begin{equation*}
(\mathfrak{M},w)\models\langle\phi\rangle\psi\text{ iff }(\mathfrak{M},w)\models\phi\text{ and }(\mathfrak{M}|\phi,w)\models\psi
\end{equation*}
Due to a theorem of Plaza \cite{Pla1989}, it is well known that public announcements in the context of classical epistemic logic do not add expressive power. Moreover, there is a process to reduce constructs of the form $[\phi]\psi$ to equivalent compounds of basic epistemic modal formulas.\\

Followingly, a major result is the axiomatization of the classical public announcement logic $\mathbf{PA}$ with a reduction-style Hilbert calculus. It immediately arises as a natural question whether this property can be recovered in the context of knowing why.\\

For the purpose of a better insight to the problem, we consider the following expansion of the semantic evaluations of a formula of the form $[\phi]Ky_a\psi$.
\begin{align*}
&\text{Let }\mathfrak{M}\text{ and }w\in W\text{ be arbitrary:}                     &&\\
&(\mathfrak{M},w)\models [\phi]Ky_a\psi                                            &&\Leftrightarrow\\
&(\mathfrak{M},w)\models\phi\text{ implies }(\mathfrak{M}|\phi ,w)\models Ky_a\psi &&\Leftrightarrow\\
&(\mathfrak{M},w)\models\phi\text{ implies }((\mathfrak{M}|\phi,w)\models K_a\psi\text{ and }\exists t\in E:\forall v\in W': (w,v)\in R'_a\text{ implies }v\in\mathcal{E}'(t,\psi ))                             &&\\
\end{align*}
From this equivalences, we suppose that there is no basic formula in the language of the logic $\mathbf{ELKy}$ which is able to express this matter equally over a reduction into a compound of more simple epistemic formulas, e.g. via pushing the announcements further into the formula, since modifying $\psi$ results in some problems concerning the integrity of $\mathcal{E}$ in the evaluation, which focuses on the syntactic structure of the concerned formula. We concretize this assumption through the following theorem.
\begin{theorem}\label{thm:pafkygreaterelky}
$\mathbf{PAFKy}$ is more expressive than $\mathbf{ELKy}$.
\end{theorem}
\begin{proof}
We sketch the model-theoretic considerations. Since $\mathbf{PAFKy}$ is an extension of $\mathbf{ELKy}$, it can't be less expressive. Now, consider the following two models $\mathfrak{M}_1$, $\mathfrak{M}_2$.\\

For $\mathfrak{M}_1$, we consider a set of worlds containing two elements $w_1,w_2$, we assume that $E$ contains two basic elements $s,t$ besides the usual conditions, $R_a$ is expected to be total for all agents $a$. For $\mathcal{E}$, we set $\mathcal{E}(s,p)=\{w_1\}, \mathcal{E}(t,p)=\{w_2\}$ while the rest are set to $\emptyset$. For the basic evaluation function $V$, we finally set $V(p)=\{w_1,w_2\}$ and $V(q)=\{w_1\}$. Visually, this model may be imagined as the following:
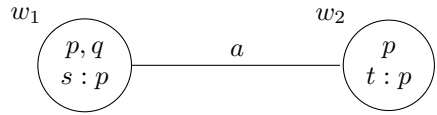
\begin{center}
\begin{tikzpicture}[shorten >=1pt,node distance=4cm,on grid,auto]
\node[state, label=above left:$w_1$, align=center] (w_1) {$p,q$ \\ $s:p$};
\node[state, label=above left:$w_2$, align=center] (w_2) [right=of w_1] {$p$ \\ $t:p$};
\path[-] (w_1) edge node {$a$} (w_2);
\end{tikzpicture}
\end{center}
Suppose that $\mathfrak{M}_1$ is a submodel of $\mathfrak{M}_2$. Therefore, we will only mention the additional settings. We consider an augmented set of worlds by a third world $w_3$, together with a third contained basic explanation $r\in E$. $R_a$ is still considered to be total for all agents and we additionally require $\mathcal{E}(r,p)=\{w_3\}$ and $V(p)=\{w_1,w_2,w_3\}$, $V(q)=\{w_1,w_3\}$. We may imagine this second model in a slightly reduced representation as:
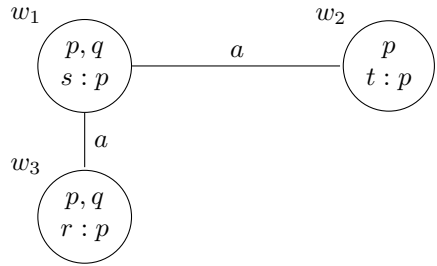
\begin{center}
\begin{tikzpicture}[shorten >=1pt,node distance=2cm and 4cm,on grid,auto]
\node[state, label=above left:$w_1$, align=center] (w_1) {$p,q$ \\ $s:p$};
\node[state, label=above left:$w_2$, align=center] (w_2) [right=of w_1] {$p$ \\ $t:p$};
\node[state, label=above left:$w_3$, align=center] (w_3) [below=of w_1] {$p,q$ \\ $r:p$};
\path[-] (w_1) edge node {$a$} (w_2)
               edge node {$a$} (w_3);
\end{tikzpicture}
\end{center}
In $\mathbf{PAFKy}$, we may distinguish these model through the formula $[q]Ky_a p$ in the designed world $w_1$ contained in both models. For $\mathfrak{M}_1$, we find that the only world left after the announcement made in $w_1$ is $w_1$ itself which provides $(\mathfrak{M}_1,w_1)\models [q]Ky_a p$. For $\mathfrak{M}_2$, we find that $W^q_2=\{w_1,w_3\}$ after the announcement. Through $\bigcap_{t\in E}\mathcal{E}_2(t,p)=\emptyset$, we find that $\not\exists t\in E: \forall v\in W^q_2: v\in\mathcal{E}(t,p)$, i.e. $(\mathfrak{M}_2,w_1)\not\models [q]Ky_a p$.\\

To show that there exists no $\phi\in\mathcal{L}_{ELKy}$ which can distinguish between the reasons from only $w_1$ and $w_3$, we suppose the opposite. If so, then $\phi = q\rightarrow \Theta$ for some $\Theta\in\mathcal{L}_{ELKy}$. By $q$ used in the implication, we exclude the case of $q$ being un-announceable in $w_1$. By the semantics of $[q]Ky_a p$, $\Theta$ would then have to model the behavior of $(\mathfrak{M}|q,w_1)\models Ky_a p$, i.e. the behavior of $Ky_a p$ restricted to all $q$-worlds.\\

By a simple induction on the structure of $\Theta$, it can be seen that no such formula exists, since to express something about the reasons of $p$ in reachable worlds, the only possibility is to include $Ky_a p$ as some subformula which itself can not be limited to some subset of worlds over the use of $\neg,\land ,K_a$ or $Ky_a$(which form the induction steps).
\end{proof}
It became apparent that although the sub-case $[\phi]K_a\psi$ is expressible over an adequate translation, we have no other possibility of expressing something concerning the explanation function $\mathcal{E}$ than the operator $Ky_a\psi$ and for this purpose, there can be no modification of $\psi$ as it would mess with the before mentioned integrity of $\mathcal{E}$ by changing the concerned formula for the existence of explanations, i.e. there is no way to restrict the application of $\mathcal{E}$ uniformly to a subset of worlds. Through this theorem, we can obviusly not apply Plaza's method from \cite{Pla1989} to provide completeness over reduction.
\subsection{A relativized knowing-why operator}
To address this problem, we're following the ideas of \cite{BEK2006}, where the authors used the concept of relativization for similar problems concerning common knowledge, and of \cite{WF2013}, \cite{WF2014} from the context of non-classical epistemic logics, by relativizing the $Ky$ operator, turning it into a conditional version, namely $Ky_a^r(\phi,\psi)$, with the following semantics:
\begin{align*}
(\mathfrak{M},w)\models Ky_a^r(\phi,\psi)&\text{ iff }\exists t\in E:\forall v\in W\text{ such that }(w,v)\in R_a\text{ and }(\mathfrak{M},v)\models\phi\text{ it is that}\\ &\text{ (1): }v\in\mathcal{E}(t,\psi)\text{ and (2): }(\mathfrak{M},v)\models\psi
\end{align*}
This operator relates to ``the agent $a$ knows why $\psi$, under the condition $\phi$``. Clearly, the original, unary, operator $Ky_a\phi$ corresponds to $Ky_a^r(\top,\phi)$.

Two new versions of the logics presented before, namely $ELKy^r$ and following to this $PAFKy^r$ are directly emerging from this, simply with
\begin{equation*}
\mathcal{L}_{ELKy^r}:\phi ::= p\; |\;\neg\phi\; |\; (\phi\land\phi)\; |\;K_a\phi\; |\;Ky_a^r(\phi,\phi)
\end{equation*}
with $p\in\mathcal{P}$ and $a\in\mathcal{A}$ and $\mathcal{L}_{PAFKy^r}$ being simply the augmentation of $\mathcal{L}_{ELKy^r}$ with the notion of the $[\phi]$ operator as shown before. The class of models associated with this new logic, named $\mathcal{K}y^r\mathcal{S}5$ for the $\mathcal{S}5$-relational version, are structurally similar to the $\mathcal{K}y\mathcal{S}5$-models presented before, while $Ky$ and its semantic evaluation in $\models$ are replaced with $Ky^r$ and the corresponding definition above.
\begin{proposition}\label{prop:axioms1}
The following formulas are valid in the class of all $\mathcal{S}5$-$ELKy^r$-models.
\begin{multicols}{2}
\begin{enumerate}[(i)]
\item $Ky_a^r(\phi,\psi\rightarrow\chi )\rightarrow (Ky_a^r(\phi,\psi)\rightarrow Ky_a^r(\phi,\chi))$
\item $Ky_a^r(\phi,\psi)\rightarrow K_a(\phi\rightarrow\psi)$
\item $Ky_a^r(\phi,\psi)\rightarrow K_aKy_a^r(\phi,\psi)$
\item $Ky_a^r(\psi,\chi)\land K_a(\phi\rightarrow\psi)\rightarrow Ky_a^r(\phi,\chi)$
\item $K_a\neg\phi\rightarrow Ky_a^r(\phi,\psi)$
\end{enumerate}
\end{multicols}
\end{proposition}
\begin{proof}
In the following proofs, let $\mathfrak{M}$ be an arbitrary model and $w\in\mathcal{D}(\mathfrak{M})$ an arbitrary world.
\begin{enumerate}[(i)]
\item Suppose $(\mathfrak{M},w)\models Ky_a^r(\phi,\psi\rightarrow\chi)$ and $(\mathfrak{M},w)\models Ky_a^r(\phi,\psi)$. By the rules of $\models$, the first one translates to $\exists t\in E:\forall v\in W: (w,v)\in R_a$ and $(\mathfrak{M},v)\models\phi$ implies $(\mathfrak{M},v)\models\psi\rightarrow\chi$ and $v\in\mathcal{E}(t,\psi\rightarrow\chi)$. The second one the translates to $\exists s\in E:\forall v\in W: (w,v)\in R_a$ and $(\mathfrak{M},v)\models\phi$ implies $(\mathfrak{M},v)\models\psi$ and $v\in\mathcal{E}(s,\psi)$. By meet over $W,R_a,E$ and $\phi$, we have: $\exists s,t\in E:\forall v\in W: (w,v)\in R_a$ and $(\mathfrak{M},w)\models\phi$ implies $v\in\mathcal{E}(t,\psi\rightarrow\chi)$, $v\in\mathcal{E}(s,\psi)$ and $(\mathfrak{M},v)\models\psi$ and $(\mathfrak{M},v)\models\psi\rightarrow\chi$. From \emph{(MP)}, we have implied that $(\mathfrak{M},v)\models\chi$. With this, and $v\in\mathcal{E}(t,\psi\rightarrow\chi)\cap\mathcal{E}(s,\psi)\subseteq\mathcal{E}(s\cdot t,\chi)$, we have $(\mathfrak{M},w)\models Ky_a^r(\phi,\chi)$.
\item Suppose $(\mathfrak{M},w)\models Ky_a^r(\phi,\psi)$. By this, we have $\exists t\in E:\forall v\in W:(w,v)\in R_a$ and $(\mathfrak{M},v)\models\phi$ implies $(\mathfrak{M},v)\models\psi$ and $v\in\mathcal{E}(t,\psi)$. By cutting out $E$, we have $\forall v\in W:(w,v)\in R_a$ and $(\mathfrak{M},v)\models\phi$ implies $(\mathfrak{M},v)\models\psi$. This implies propositionally, that $\forall v\in W:(w,v)\in R_a$ implies $((\mathfrak{M},v)\models\phi$ implies $(\mathfrak{M},v)\models\psi)$, i.e. $(\mathfrak{M},w)\models K_a(\phi\rightarrow\psi)$.
\item Suppose that $(\mathfrak{M},w)\models Ky_a^r(\phi,\psi)$. Now, we consider any $v\in W$ with $(w,v)\in R_a$ and $u\in W$ with $(v,u)\in R_a$. By the transitivity of the relations, we have $(w,u)\in R_a$ implied. By $(w,u)\in R_a$ for every such $u$, we have $\exists t\in E:\forall u\in W:(v,u)\in R_a$ and $(\mathfrak{M},u)\models\phi$ implies $(\mathfrak{M},u)\models\psi$ and $u\in\mathcal{E}(t,\psi)$, therefore $(\mathfrak{M},v)\models Ky_a^r(\phi,\psi)$ for every such $v$ where $(w,v)\in R_a$, and thus we have $(\mathfrak{M},w)\models K_aKy_a^r(\phi,\psi)$.
\item Suppose that $(\mathfrak{M},w)\models Ky_a^r(\psi,\chi)$ and $(\mathfrak{M},w)\models K_a(\phi\rightarrow\psi)$, i.e. $\forall v\in W: (w,v)\in R_a$ impl. $(\mathfrak{M},v)\models\phi\rightarrow\psi$. By the first, there exists a $t\in E$ s.t. $\forall v\in W:(w,v)\in R_a$ and $(\mathfrak{M},v)\models\psi$ impl. $(\mathfrak{M},v)\models\chi$ and $v\in\mathcal{E}(t,\chi)$. Take this $t$, and any $v\in W$ s.t. $(w,v)\in R_a$ and $(\mathfrak{M},v)\models\phi$. Thus, $(\mathfrak{M},v)\models\psi$ and thus $v\in\mathcal{E}(t,\chi)$ and $(\mathfrak{M},v)\models\chi$. Thus $(\mathfrak{M},w)\models Ky_a^r(\phi,\chi)$ by taking the meet over $W$.
\item Suppose that $\forall v\in W:(w,v)\in R_a$ impl. $(\mathfrak{M},v)\models\neg\phi$, i.e. $(\mathfrak{M},v)\not\models\phi$. Thus, there is no world s.t. $(w,v)\in R_a$ and $(\mathfrak{M},v)\models\phi$, i.e. $\forall v\in W:(w,v)\in R_a$ and $(\mathfrak{M},v)\models\phi$ implies $(\mathfrak{M},v)\models\psi$ and thus $\exists t\in E:\forall v\in W:(w,v)\in R_a$ and $(\mathfrak{M},v)\models\phi$ implies $(\mathfrak{M},v)\models\psi$ and $v\in\mathcal{E}(t,\psi)$ as $E$ is non-empty. Thus $(\mathfrak{M},w)\models Ky_a^r(\phi,\psi)$.
\end{enumerate}
\end{proof}
We can even consider a stronger version of the distribution of the $Ky^r$ operator over $\rightarrow$ in the right argument by taking different premises into account.
\begin{proposition}\label{prop:ckyr}
$Ky_a^r(\chi,\phi\rightarrow\psi)\land Ky_a^r(\theta,\phi)\rightarrow Ky_a^r(\chi\land\theta, \psi)$ is valid in $\mathcal{K}y^r\mathcal{S}5$.
\end{proposition}
\begin{proof}
Suppose $(\mathfrak{M},w)\models Ky_a^r(\chi,\phi\rightarrow\psi)$ and $(\mathfrak{M},w)\models Ky_a^r(\theta,\phi)$. The first one translates to $\exists t\in E:\forall v\in W: (w,v)\in R_a$ and $(\mathfrak{M},v)\models\chi$ implies $(\mathfrak{M},v)\models\phi\rightarrow\psi$ and $v\in\mathcal{E}(t,\phi\rightarrow\psi)$, while the second translates to $\exists s\in E:\forall v\in W: (w,v)\in R_a$ and $(\mathfrak{M},v)\models\theta$ implies $(\mathfrak{M},v)\models\phi$ and $v\in\mathcal{E}(s,\phi)$. By meet over $W$, $R_a$ and $E$, we have $\exists t,s\in E:\forall v\in W:(w,v)\in R_a$ and $(\mathfrak{M},v)\models\chi\land\theta$ implies $(\mathfrak{M},v)\models\phi\rightarrow\psi$ and $(\mathfrak{M},v)\models\phi$ and $v\in\mathcal{E}(t,\phi\rightarrow\psi)\cap\mathcal{E}(s,\phi)$. By the laws of modus ponens, we have $(\mathfrak{M},v)\models\psi$ for the latter and by the properties of $\mathcal{E}$, we have $\mathcal{E}(t,\phi\rightarrow\psi)\cap\mathcal{E}(s,\phi)\subseteq\mathcal{E}(t\cdot s,\psi)$, therefore $(\mathfrak{M},w)\models Ky_a^r(\chi\land\theta,\psi)$.
\end{proof}
If two different premises $\chi$ and $\theta$ in the formulas $Ky_a^r(\chi,\phi\rightarrow\psi)$ and $Ky_a^r(\theta,\phi)$ are actually contradictory to each other, e.g. $p$ and $\neg p$, the formula $Ky_a(\chi\land\theta,\psi)$ automatically relates to the validity $Ky_a^r(\bot,\phi)$.\\

The validity \emph{(i)} from Prop. \ref{prop:axioms1} is obviously semantically a special case of Prop. \ref{prop:ckyr} with both $\chi$ and $\theta$ representing the same formula.
\subsection{An axiomatization of $ELKy^r$}
As an axiomatization for the logic $\mathbf{ELKy^r}$, we propose the here shown Hilbert-calculus as an adaption of the system $\mathbb{SKY}$ from the initial paper to the new notion of the relativized knowing-why operator. In correspondence to the before-mentioned paper, we call this system $\mathbb{SKYR}$.
\begin{definition}[The system $\mathbb{SKYR}$]
The proof system $\mathbb{SKYR}$ is defined as the following Hilbert-style calculus:
\begin{description}[leftmargin=!,labelwidth=\widthof{($(DKyR)$):}]
\item [$(PT)$] $\text{the classical propositional axioms}$
\item [$(K)$] $K_a(\phi\rightarrow\psi )\rightarrow (K_a\phi\rightarrow K_a\psi)$
\item [$(T)$] $K_a\phi\rightarrow\phi$
\item [$(4)$] $K_a\phi\rightarrow K_aK_a\phi$
\item [$(5)$] $\neg K_a\phi\rightarrow K_a\neg K_a\phi$
\item [$(EKyR)$] $Ky_a^r(\chi,\phi\rightarrow\psi)\rightarrow(Ky_a^r(\theta,\phi)\rightarrow Ky_a^r(\chi\land\theta,\psi))$
\item [$(4YKR)$] $Ky_a^r(\phi,\psi)\rightarrow K_aKy_a^r(\phi,\psi)$
\item [$(DKyR)$] $Ky_a^r(\phi,\psi)\rightarrow K_a(\phi\rightarrow\psi)$
\item [$(IKyR)$] $Ky_a^r(\psi,\chi)\rightarrow(K_a(\phi\rightarrow\psi)\rightarrow Ky_a^r(\phi,\chi))$
\item [$(UKyR)$] $K_a\neg\phi\rightarrow Ky_a^r(\phi,\psi)$
\item [$(MP)$] $\text{From }\phi\text{ and }\phi\rightarrow\psi\text{, infer }\psi$
\item [$(NK)$] $\text{From}\vdash\phi\text{, infer}\vdash K_a\phi$
\item [$(NKyR)$] $\text{From }\phi\in\Lambda\text{, infer}\vdash Ky_a^r(\top,\phi)$
\end{description}
\end{definition}
The axiom \emph{(EKyR)} provides, as mentioned before, a stronger version of the distribution of $Ky^r$ over $\rightarrow$ in the right argument by using different premises in the left argument. The axiom \emph{(DKyR)} defines the decomposition or extraction of the fragment concerning the basic knowledge operator $K$ from $Ky^r$ and axiom \emph{(4YKR)} provides the positive introspection of $Ky^r$ by the classical operator $K$. \emph{(IKyR)} allows inference of knowing why for stronger premises, provided the case for the weaker premise is already established and \emph{(UKyR)} describes the situation if the condition of the $Ky^r$-operator is \emph{impossible} from a current world, i.e. if every reachable world does not satisfy the condition.\\

One may wonder about the axiom \emph{(5YKR)}, the negative introspection of $Ky^r$ by the operator $K$, which is actually provable, exactly as its unconditioned companionen \emph{(5YK)} was in the basic system $\mathbb{SKY}$(s. Proposition 11, \cite{XWS2016}).
\begin{proposition}
The following formulas are provable in $\mathbb{SKYR}$:
\begin{enumerate}
\item $\neg Ky_a^r(\phi,\psi)\rightarrow K_a\neg Ky_a^r(\phi,\psi)$
\item $Ky_a^r(\bot,\phi)$
\item $Ky_a^r(\phi,\psi\rightarrow\chi)\rightarrow (Ky_a^r(\phi,\psi)\rightarrow Ky_a^r(\phi,\chi))$
\end{enumerate}
\end{proposition}
\begin{proof}
In the following, although \emph{(PT)} references the usual propositional axioms, we let it denote any theorem of the basic classical propositional calculus (over this new language).
\begin{enumerate}
\item As a line derivation: \begin{lstlisting}[mathescape=true]
$K_aKy_a^r(\phi,\psi)\rightarrow Ky_a^r(\phi,\psi)$ &\Comment{Instance of \emph{(T)}}&
$\neg Ky_a^r(\phi,\psi)\rightarrow \neg K_aKy_a^r(\phi,\psi)$ &\Comment{Contraposition of 1}&
$\neg K_aKy_a^r(\phi,\psi)\rightarrow K_a\neg K_aKy_a^r(\phi,\psi)$ &\Comment{Instance of \emph{(5)}}&
$Ky_a^r(\phi,\psi)\rightarrow K_aKy_a^r(\phi,\psi)$ &\Comment{Instance of \emph{(4YKR)}}&
$\neg K_aKy_a^r(\phi,\psi)\rightarrow\neg Ky_a^r(\phi,\psi)$ &\Comment{Contraposition of 4}&
$K_a(\neg K_aKy_a^r(\phi,\psi)\rightarrow\neg Ky_a^r(\phi,\psi))$ &\Comment{\emph{(NK)} on 5}&
$K_a\neg K_aKy_a^r(\phi,\psi)\rightarrow K_a\neg Ky_a^r(\phi,\psi)$ &\Comment{\emph{(MP)} with \emph{(K)} on 6}&
$\neg Ky_a^r(\phi,\psi)\rightarrow K_a\neg K_aKy_a^r(\phi,\psi)$ &\Comment{\emph{(MP)} on 2,3}&
$\neg Ky_a^r(\phi,\psi)\rightarrow K_a\neg Ky_a^r(\phi,\psi)$ &\Comment{\emph{(MP)} on 7,8}&
\end{lstlisting}
\item As a line derivation: \begin{lstlisting}[mathescape=true]
$\neg\bot$ &\Comment{Instance of \emph{(PT)}}&
$K_a\neg\bot$ &\Comment{\emph{(NK)} on 1}&
$K_a\neg\bot\rightarrow Ky_a^r(\bot,\phi)$ &\Comment{Instance of \emph{(UKyR)}}&
$Ky_a^r(\bot,\phi)$ &\Comment{\emph{(MP)} on 2,3}&
\end{lstlisting}
\item As a line derivation: \begin{lstlisting}[mathescape=true]
$Ky_a^r(\phi,\psi\rightarrow\chi)\rightarrow(Ky_a^r(\phi,\psi)\rightarrow Ky_a^r(\phi\land\phi,\chi))$ &\Comment{Instance of \emph{(EKyR)}}&
$Ky_a^r(\phi,\psi\rightarrow\chi)\land Ky_a^r(\phi,\psi)\rightarrow Ky_a^r(\phi\land\phi,\chi)$ &\Comment{\emph{(MP)} and \emph{(PT)} on 1}&
$\phi\rightarrow(\phi\land\phi)$ &\Comment{Instance of \emph{(PT)}}&
$K_a(\phi\rightarrow(\phi\land\phi))$ &\Comment{\emph{(NK)} on 3}&
$Ky_a^r(\phi\land\phi,\chi)\rightarrow (K_a(\phi\rightarrow(\phi\land\phi))\rightarrow Ky_a^r(\phi,\chi))$ &\Comment{Instance of \emph{(IKyR)}}&
$K_a(\phi\rightarrow(\phi\land\phi))\rightarrow (Ky_a^r(\phi\land\phi,\chi)\rightarrow Ky_a^r(\phi,\chi))$ &\Comment{\emph{(MP)} and \emph{(PT)} on 5}&
$Ky_a^r(\phi\land\phi,\chi)\rightarrow Ky_a^r(\phi,\chi)$ &\Comment{\emph{(MP)} on 4,6}&
$Ky_a^r(\phi,\psi\rightarrow\chi)\land Ky_a^r(\phi,\psi)\rightarrow Ky_a^r(\phi,\chi)$ &\Comment{\emph{(MP)} on 2,7}&
$Ky_a^r(\phi,\psi\rightarrow\chi)\rightarrow (Ky_a^r(\phi,\psi)\rightarrow Ky_a^r(\phi,\chi))$ &\Comment{\emph{(MP)} and \emph{(PT)} on 8}&
\end{lstlisting}
\end{enumerate}
\end{proof}
From this point on, we can almost immediately consider the soundness of our system $\mathbb{SKYR}$ by finally proposing the following.
\begin{lemma}\label{lem:rlnkyr}
The rule \emph{(NKyR)} is valid.
\end{lemma}
\begin{proof}
Let $\phi\in\Lambda$ and $\mathfrak{M}$ be any model. Since $\Lambda$ only contains tautologies, we have $(\mathfrak{M},w)\models\phi$ f.a. $w\in W$. By the first rule established for the behavior of $\mathcal{E}$, we have that $\exists e\in E:\mathcal{E}(e,\phi)= W$. Thus, for every $w\in W$, we have that $(\mathfrak{M},w)\models\phi$ and $w\in\mathcal{E}(e,\phi)$, i.e. we have $\exists t\in E:\forall v\in W: (w,v)\in R_a$ and $(\mathfrak{M},v)\models\top$ implies $(\mathfrak{M},v)\models\phi$ and $v\in\mathcal{E}(t,\phi)$, i.e. $(\mathfrak{M},w)\models Ky_a^r(\top,\phi)$ and this for any $w\in W$.
\end{proof}
\begin{lemma}
The generalized necessitation rule for $Ky_a^r$, i.e. 
\[\text{From }\phi\in\Lambda\text{, infer }\vdash Ky_a^r(\psi,\phi)\]
is admissible for any $\psi\in\mathcal{L}_{ELKy^r}$.
\end{lemma}
\begin{proof}
Let $\phi\in\Lambda$ and $\psi\in\mathcal{L}_{ELKy^r}$. By \emph{(NKyR)}, we have $\vdash Ky_a^r(\top,\phi)$. As an instance of a propositional tautology, we have $\vdash\psi\rightarrow\top$ and by \emph{(NK)}, $\vdash K_a(\psi\rightarrow\top)$. By \emph{(IKyR)} and \emph{(MP)}, we have $\vdash Ky_a^r(\psi,\phi)$.
\end{proof}
\begin{theorem}[Soundness of $\mathbb{SKYR}$ over $\mathcal{K}y^r\mathcal{S}5$]
The system $\mathbb{SKYR}$ is sound with respect to the class of all $\mathcal{S}5-ELKy^r$-models.
\end{theorem}
\begin{proof}
This theorem is easily obtained by considering Lem. \ref{lem:rlnkyr}, Prop. \ref{prop:axioms1} and Prop. \ref{prop:ckyr} together with the fact, that the $\mathbf{ELKy^r}$-models are based on the standard \emph{Kripkean}-$\mathcal{S}5$-models, making both the standard $\mathcal{S}5$ axioms and the rule \emph{(NK)} valid.
\end{proof}
\subsubsection{Completeness}
For the proof of completeness' sake, we follow the common approach of considering a \emph{canonical} model being defined over all maximal consistent sets with the(later more explicitly defined) common properties like \emph{truth} in order to provide the framework for a somehow \emph{standard} proof of a completeness theorem in the context of modal logics. For this, we first consider:
\begin{definition}[Consistency]
A set $\Gamma\subseteq\mathcal{L}_{ELKy^r}$ is called \emph{consistent} (in $\mathbb{SKYR}$), if $\Gamma\not\vdash_\mathbb{SKYR}\bot$. Otherwise, it is called inconsistent, i.e. with $\Gamma\vdash_\mathbb{SKYR}\bot$. Following to this, a set is called maximal consistent (over $\mathcal{L}_{ELKy^r}$) if
\begin{enumerate}
\item it is consistent, i.e. $\Gamma\not\vdash_\mathbb{SKYR}\bot$,
\item it is maximal, i.e. $\forall\Gamma'\subseteq\mathcal{L}_{ELKy^r}:\Gamma'\supset\Gamma\text{ implies }\Gamma'\vdash\bot$.
\end{enumerate}
\end{definition}
\begin{proposition}[Properties of maximal consistent sets]\label{prop:max_con_set_prop}
Let $\Gamma$ be maximal consistent. Then for all $\phi\in\mathcal{L}_{ELKy^r}$
\begin{multicols}{2}
\begin{enumerate}[(i)]
\item $\Gamma\vdash\phi$ iff $\phi\in\Gamma$ (deductive closure)
\item $\phi\in\Gamma$ iff $\neg\phi\not\in\Gamma$
\item $\phi,\psi\in\Gamma$ iff $(\phi\land\psi)\in\Gamma$
\item $\top\in\Gamma$
\end{enumerate}
\end{multicols}
\end{proposition}
\begin{proposition}[Lindenbaum]\label{prop:lindenbaum}
Every consistent set can be extended to a maximal consistent set.
\end{proposition}
The proof of both propositions is very canonical and thus omitted here. Following from these considerations, we now define a canonical model with worlds corresponding to maximal consistent sets in the system $\mathbb{SKYR}$.
\begin{definition}[Canonical model for $\mathbb{SKYR}$]\label{def:canonmod}
The canonical model for $\mathbb{SKYR}$ is defined as the structure 
\begin{equation*}
\mathfrak{M}^c =\langle W^c, E^c,\{ R_a^c\;|\;a\in\mathcal{A}\},\mathcal{E}^c, V^c\rangle
\end{equation*}
with
\begin{itemize}
\item $\hat{E}^c:t::= e\mid\phi\mid(t\cdot t)$ with $\phi\in\mathcal{L}_{ELKy^r}$. We then set $E^c=\{e_\top\}\cup\{t_\phi\mid t\in\hat{E}^c\setminus\{e\}, \phi\in\mathcal{L}_{ELKy^r}\}$. $\cdot: E^c\times E^c\to E^c$ is given with $(t_\phi\cdot s_\psi)=(t\cdot s)_{\phi\land\psi}$.\footnote{Note, that we do not explicitly differentiate in notation between $\cdot$ as an operation on $E^c$ and as a syntactical connective in $\hat{E}^c$.}
\item $W^c=\{\langle\Gamma,F,\vec{f}\rangle\}$, such that $\Gamma\subseteq\mathcal{L}_{ELKy^r}$ is maximal consistent, $F\subseteq E^c\times\mathcal{L}_{ELKy^r}$, $\vec{f}=(f^\phi_a)_{a\in\mathcal{A}, \phi\in\mathcal{L}_{ELKy^r}}$, $f_a^\phi:\{\psi\mid Ky_a^r(\phi,\psi)\in\Gamma\}\to E^c$ which fulfill the following conditions:\begin{enumerate}
\item $(s_\alpha,\phi\rightarrow\psi), (r_\beta,\phi)\in F$ impl. $((s_\alpha\cdot r_\beta),\psi)=((s\cdot r)_{\alpha\land\beta},\psi)\in F$ 
\item If $\phi\in\Lambda$, then $(e_\top,\phi)\in F$ .
\item For any $a\in\mathcal{A}$, if $Ky_a^r(\phi,\psi)\land\phi\in\Gamma$, then $(f_a^\phi(\psi),\psi)\in F$.
\item For any $a\in\mathcal{A}$, any $\phi\in\mathcal{L}_{ELKy^r}$ and any $\psi\in\mathsf{dom}(f_a^\phi)$, $f_a^\phi(\psi)=t_\phi$ for some $t\in \hat{E}^c$.\footnote{Note, that here $t=e$ is obviously only possible if $\phi=\top$, as $t_\phi=f_a^\phi(\psi)\in E^c$}
\end{enumerate}
\item $R_a^c=\{(\langle\Gamma,F,\vec{f}\rangle,\langle\Delta,G,\vec{g}\rangle)\in W^c\times W^c\mid\Gamma^\#_a=\{\phi\mid K_a\phi\in\Gamma\}\subseteq\Delta, \forall\phi\in\mathcal{L}_{ELKy^r}:f_a^\phi=g_a^\phi\}$, for any $a\in\mathcal{A}$
\item $\mathcal{E}^c(t_\phi,\psi)=\{\langle\Gamma,F,\vec{f}\rangle\in W^c\mid(t_\phi,\psi)\in F\}$, for $t_\phi\in E^c$
\item $V^c(p)=\{\langle\Gamma,F,\vec{f}\rangle\in W^c\mid p\in\Gamma\}$, for $p\in\mathcal{P}$
\end{itemize}
\end{definition}
Note, that for the canonical model, $e_\top$ represents the required constant $e$ from the original definition, indexed by $\top$ here just to be in line with the notation. This definition is non-degenerate, as $e_\phi$ for $\phi\neq\top$ is non-existing in $E^c$. The design choices have been made in coorperation to the basic ideas of \cite{XWS2016} about the there presented canonical model for $\mathbb{SKY}$.

The main difference here is that we conditionalized the explanations, forcing them to keep track under which conditional formulas they were enforced to be explanations for some formula. For further procedure, we first need to show that the model itself is well defined according to the specifications shown in the preliminaries.\\

As it was observed in \cite{XWS2016}, a set of worlds with a 1-to-1 correspondence to maximal consistent sets is not sufficient to provide the classical behavior of a canonical model in the context of the knowing-why notions, as there is not a canonical way to attach explanation to formulas of a given maximal consistent set, i.e. setting $\mathcal{E}^c$ at the respectively associated world, and thus it may be desirable to have the same maximal consistent set accompanied by various versions of possible explanation scenarios. This stands in contrast to the completeness proof of classical justification logic, where a simple 1-to-1 correspondence between maximal consistent sets and worlds in the canonical model is possible. Mainly, in the newly shedded light from the logic of knowing why, this result from justification logic directly encoding the used explanations into the syntax, thus in a way fixing the desired state for pairs of explanations and formulas through the maximal consistent set of formulas directly.\\

Before proceeding to the proof of the well-definedness of $\mathfrak{M}^c$, we first adapt a helpful proposition  from \cite{XWS2016} concerning the accessibility relations.
\begin{proposition}\label{prop:accesprop}
Let $\Gamma,\Delta\subseteq\mathcal{L}_{ELKy^r}$ be max. consistent. If $\Gamma^\#_a\subseteq\Delta$, then
\begin{multicols}{2}
\begin{enumerate}[(i)]
\item $K_a\phi\in\Gamma$ iff $K_a\phi\in\Delta$
\item $Ky_a^r(\phi,\psi)\in\Gamma$ iff $Ky_a^r(\phi,\psi)\in\Delta$
\end{enumerate}
\end{multicols}
\end{proposition}
\begin{proof}
Using properties for maximal consistent sets (Prop. \ref{prop:max_con_set_prop}), we infer:
\begin{enumerate}[(i)]
\item Let $K_a\phi\in\Gamma$. By axiom \emph{(4)}(and deductive closure of $\Gamma$\footnote{This phrase will be omitted in the following.}), we have $K_aK_a\phi\in\Gamma$, i.e. $K_a\phi\in\Delta$ per definition of $\Gamma^\#_a\subseteq\Delta$. Let $K_a\phi\not\in\Gamma$, i.e. $\neg K_a\phi\in\Gamma$ and by axiom \emph{(5)}, we have $K_a\neg K_a\phi\in\Gamma$, i.e. $\neg K_a\phi\in\Delta$, i.e. $K_a\phi\not\in\Delta$.
\item Let $Ky_a^r(\phi,\psi)\in\Gamma$, i.e. $K_aKy_a^r(\phi,\psi)\in\Gamma$ by \emph{(4YKR)}, thus $Ky_a^r(\phi,\psi)\in\Delta$. Let $Ky_a^r(\phi,\psi)\not\in\Gamma$, i.e. $\neg Ky_a^r(\phi,\psi)\in\Gamma$, i.e. by \emph{(5YKR)}, we have $K_a\neg Ky_a^r(\phi,\psi)\in\Gamma$, i.e. $\neg Ky_a^r(\phi,\psi)\in\Delta$, i.e. $Ky_a^r(\phi,\psi)\not\in\Delta$.
\end{enumerate}
\end{proof}
For the following, it is also interesting to note that by the above for two $\langle\Gamma,F,\vec{f}\rangle,\langle\Delta,G,\vec{g}\rangle\in W^c$, if $\Gamma^\#_a\subseteq\Delta$, then $\mathsf{dom}(f_a^\phi)=\mathsf{dom}(g_a^\phi)$ for any $\phi\in\mathcal{L}_{ELKy^r}$.
We now propose:
\begin{proposition}
The canonical model $\mathfrak{M}^c$ for $\mathbb{SKYR}$ is well defined, given the conditions for $\mathcal{S}5-ELKy^r$($\mathcal{S}5-ELKy$)-models.
\end{proposition}
\begin{proof}
First, we check that $E^c$ is well-defined. For this, note that $E^c$ holds a designated explanation $e_\top$ whose circumstances were explained before. Also, $E^c$ is closed under the combination operation $\cdot$. To see this, let $x,y\in E^c$, i.e. $x=t_\phi,y=s_\psi$ for
\begin{itemize}
\item $\phi,\psi\in\mathcal{L}_{ELKy^r}$ and $t,s\in \hat{E}^c\setminus\{e\}$, or
\item $t=e$ and thus $\phi=\top$, or
\item $s=e$ and thus $\psi=\top$.
\end{itemize}
In any case, we have $(x\cdot y)=(t_\phi\cdot s_\psi)=(t\cdot s)_{\phi\land\psi}$ and in any case we have $t,s\in\hat{E}^c$(as obviously also $e\in\hat{E}^c$). Thus, $(t\cdot s)\in\hat{E}^c$ and as never $e=(t\cdot s)$, we have $(t\cdot s)\in\hat{E}^c\setminus\{e\}$. Thus $(t\cdot s)_{\phi\land\psi}=(t_\phi\cdot s_\psi)\in E^c$ per definition.\\

Now we check the conditions on $\mathcal{E}^c$ and $R_a^c$ for any $a\in\mathcal{A}$:
\begin{description}
\item [$\mathcal{E}^c(e_\top,\phi)= W^c$, $\phi\in\Lambda$] Let $\phi\in\Lambda$, thus by def. of $\mathfrak{M}^c$, we have $(e_\top,\phi)\in F$ for every $\langle\Gamma,F,\vec{f}\rangle\in W^c$, i.e. $\langle\Gamma,F,\vec{f}\rangle\in\mathcal{E}^c(e_\top,\phi)$ f.a. $\langle\Gamma,F,\vec{f}\rangle\in W^c$, i.e. $\mathcal{E}^c(e_\top,\phi)=W^c$.
\item [$\mathcal{E}^c(s_\alpha,\phi\rightarrow\psi)\cap\mathcal{E}^c(r_\beta,\phi)\subseteq\mathcal{E}^c((s\cdot r)_{\alpha\land\beta},\psi)$] Let $\langle\Gamma,F,\vec{f}\rangle\in\mathcal{E}^c(s_\alpha,\phi\rightarrow\psi)\cap\mathcal{E}^c(r_\beta,\phi)$, i.e. $(s_\alpha,\phi\rightarrow\psi),(r_\beta,\phi)\in F$. Thus, per def., $((s\cdot r)_{\alpha\land\beta},\psi)\in F$, i.e. $\langle\Gamma,F,\vec{f}\rangle\in\mathcal{E}^c((s\cdot r)_{\alpha\land\beta},\psi)$.
\item [$R_a^c$ is reflexive] Obviously, $f_a^\phi=f_a^\phi$ f.a. $\phi\in\mathcal{L}_{ELKy^r}$. Also, by \emph{(T)}, for every $K_a\phi\in\Gamma$, we have $\phi\in\Gamma$, i.e. $\Gamma^\#_a\subseteq\Gamma$.
\item [$R_a^c$ is symmetric] Supp. $(\langle\Gamma,F,\vec{f}\rangle,\langle\Delta,G,\vec{g}\rangle)\in R_a^c$, i.e. $\Gamma^\#_a\subseteq\Delta$ and $f_a^\phi=g_a^\phi$ f.a. $\phi\in\mathcal{L}_{ELKy^r}$. Automatically, $g_a^\phi=f_a^\phi$ f.a. $\phi\in\mathcal{L}_{ELKy^r}$. Also, let $K_a\phi\in\Delta$, then by Prop. \ref{prop:accesprop}, we have $K_a\phi\in\Gamma$, i.e. by \emph{(T)}, $\phi\in\Gamma$, i.e. $\Delta^\#_a\subseteq\Gamma$. Thus $(\langle\Delta,G,\vec{g}\rangle,\langle\Gamma,F,\vec{f}\rangle)\in R_a^c$.
\item [$R_a^c$ is transitive] Let $(\langle\Gamma,F,\vec{f}\rangle,\langle\Delta,G,\vec{g}\rangle)\in R_a^c$ and $(\langle\Delta,G,\vec{g}\rangle,\langle\Theta,H,\vec{h}\rangle)\in R_a^c$. Thus $f_a^\phi=g_a^\phi$ f.a. $\phi\in\mathcal{L}_{ELKy^r}$ and $g_a^\phi=h_a^\phi$ f.a. $\phi\in\mathcal{L}_{ELKy^r}$, i.e. $f_a^\phi=h_a^\phi$ f.a. $\phi\in\mathcal{L}_{ELKy^r}$. Also, let $K_a\phi\in\Gamma$, i.e. by \emph{(4)}, we have $K_aK_a\phi\in\Gamma$, i.e. $K_a\phi\in\Delta$ by $\Gamma^\#_a\subseteq\Delta$ and thus $\phi\in\Theta$ by $\Delta^\#_a\subseteq\Theta$. Thus $\Gamma^\#_a\subseteq\Theta$. Therefore $(\langle\Gamma,F,\vec{f}\rangle,\langle\Theta,H,\vec{h}\rangle)\in R_a^c$
\end{description}
\end{proof}
In order to fully proof the functioning of the canonical model, we have left to show that $W^c$ is not empty. Following \cite{XWS2016}, we provide a construction of some corresponding standard $F$ and $\vec{f}$ for a given $\Gamma$, by that showing that there exists at least one world $\langle\Gamma, F,\vec{f}\rangle$ for every maximal consistent set $\Gamma$ in the language of $\mathcal{L}_{ELKy^r}$ and the calculus $\mathbb{SKYR}$.
\begin{definition}\label{def:canonworldconst}
Given any maximal consistent set $\Gamma$, construct corresponding standard world parts $F^\Gamma$ and $\vec{f}^\Gamma$ as follows:
\begin{enumerate}
\item $F_0^\Gamma=\{(e_\top,\phi)\mid\phi\in\Lambda\}\cup\{(\psi_\phi,\psi)\mid\exists b\in\mathcal{A}: Ky_b^r(\phi,\psi)\land\phi\in\Gamma\}$.
\item $F_{n+1}^\Gamma=F_n^\Gamma\cup\{((s\cdot r)_{\alpha\land\beta},\psi)\mid (s_\alpha,\phi\rightarrow\psi),(r_\beta,\phi)\in F_n^\Gamma\}$, f.a. $n\geq 0$
\item $F^\Gamma=\bigcup_{n\in\mathbb{N}}F_n^\Gamma$
\item $f_b^\phi(\psi)=\psi_\phi$ f.a. $b\in\mathcal{A}$ and all $\phi\in\mathcal{L}_{ELKy^r}$, $\psi\in\mathsf{dom}(f_b^\phi)$
\end{enumerate}
\end{definition}
From this construction, we can now consider the following proposition providing the non-emptiness of $W^c$.
\begin{proposition}\label{prop:worldexistence}
For any maximal consistent set $\Gamma$, $\langle\Gamma ,F^\Gamma,\vec{f^\Gamma}\rangle\in W^c$.
\end{proposition}
\begin{proof}
It is already supposed that $\Gamma$ is maximal consistent. We show the properties \emph{(1)} - \emph{(4)} of $W^c$:
\begin{enumerate}
\item Suppose $(s_\alpha,\phi\rightarrow\psi),(r_\beta,\phi)\in F^\Gamma$, i.e. $\exists k\in\mathbb{N}:(s_\alpha,\phi\rightarrow\psi),(r_\beta,\phi)\in F^\Gamma_k$. Then by const., we have $((s\cdot r)_{\alpha\land\beta},\psi)\in F_{k+1}^\Gamma\subseteq F^\Gamma$.
\item Let $\phi\in\Lambda$, then $(e_\top,\phi)\in F_0^\Gamma\subseteq F^\Gamma$.
\item Let $a\in\mathcal{A}$ and let $Ky_a^r(\phi,\psi)\land\phi\in\Gamma$. Then $f_a^\phi(\psi)=\psi_\phi$ and $(\psi_\phi,\psi)\in F_0^\Gamma$, i.e. $(f_a^\phi(\psi),\psi)\in F^\Gamma$.
\item  Obviously, $f_a^\phi(\psi)=\psi_\phi=t_\phi$ f.s. $t\in\hat{E}^c$.
\end{enumerate}
Note that $f_a^\phi$ is well defined for every $\phi\in\mathcal{L}_{ELKy^r}$ and every $a\in\mathcal{A}$ (that is $f_a^\phi(\psi)\in E^c$ for every $\psi\in\mathsf{dom}(f_a^\phi)$) as $\psi\in\hat{E}^c\setminus\{e\}$ and thus $\psi_\phi\in E^c$, for any $\phi\in\mathcal{L}_{ELKy^r}$.
\end{proof}
In the following, we will now reestablish the existence lemmas for both $K$ and $Ky^r$ following the ideas of \cite{XWS2016} in order provide the last necessary steps before considering the truth lemma. The key of both existence lemmas is to provide constructions of worlds related by an accessibility relation which refute either the formula itself or any possible explanation in some way, provided that the corresponding $K$ or $Ky^r$ formula is not member of the to-speak set.
\begin{lemma}\label{lem:neighbourconst}
For any $\langle\Gamma,F,\vec{f}\rangle\in W^c$ and any max. cons. $\Delta$ s.t. $\Gamma^\#_a\subseteq\Delta$, there exists a $\langle\Delta,G,\vec{g}\rangle\in W^c$ with $(\langle\Gamma,F,\vec{f}\rangle,\langle\Delta,G,\vec{g}\rangle)\in R_a^c$.
\end{lemma}
\begin{proof}
Let $a\in\mathcal{A}$ be fixed and $\Gamma_a^\#\subseteq\Delta$. Note first, that thus $\mathsf{dom}(f_a^\phi)=\mathsf{dom}(g_a^\phi)$ for any $\phi\in\mathcal{L}_{ELKy^r}$. We define the world $\langle\Delta, G, \vec{g}\rangle$ as the following:
\begin{enumerate}
\item $g_b^\phi(\psi)=\begin{cases}f_b^\phi(\psi),&\text{if }b=a\\ \psi_\phi,&\text{otherwise}\end{cases}$, for any $b\in\mathcal{A}$, $\phi\in\mathcal{L}_{ELKy^r}$ and $\psi\in\mathsf{dom}(g_b^\phi)$
\item $G_0=F\cup\{(g_b^\phi(\psi),\psi)\mid Ky_b^r(\phi,\psi)\land\phi\in\Delta\}$
\item $G_{n+1}=G_n\cup\{((s\cdot r)_{\alpha\land\beta},\psi)\mid(s_\alpha,\phi\rightarrow\psi),(r_\beta,\phi)\in G_n\}$ f.a. $n\geq 0$
\item $G=\bigcup_{n\in\mathbb{N}}G_n$
\end{enumerate}
It is still left to show that the by that constructed world follows the conditions of the canonical model.\\
\begin{claim}
$\langle\Delta,G,\vec{g}\rangle\in W^c$
\end{claim}
\begin{claimproof}
By supposition, $\Delta$ is maximal consistent. We check the properties \emph{(1)} - \emph{(4)} of $W^c$:
\begin{enumerate}
\item Suppose $(s_\alpha,\phi\rightarrow\psi),(r_\beta,\phi)\in G$, i.e. $\exists k\in\mathbb{N}:(s_\alpha,\phi\rightarrow\psi),(r_\beta,\phi)\in G_k$. Then by const., we have $((s\cdot r)_{\alpha\land\beta},\psi)\in G_{k+1}\subseteq G$.
\item Let $\phi\in\Lambda$, then $(e_\top,\phi)\in F\subseteq G_0\subseteq G$.
\item Let $Ky_b^r(\phi,\psi)\land\phi\in\Delta$ for some $b\in\mathcal{A}$. Then by def. $(g_b^\phi(\psi),\psi)\in G_0\subseteq G$.
\item Let $b\in\mathcal{A}$. Obviously, $g_b^\phi(\psi)=\psi_\phi=t_\phi$ f.s. $t\in\hat{E}^c\setminus\{e\}$ if $b\neq a$. For $b=a$, as $f_a^\phi$ is well-defined, we have $g_a^\phi(\psi)=f_a^\phi(\psi)=t_\phi$ f.s. appropriate $t\in\hat{E}^c$.
\end{enumerate}
\end{claimproof}
We have $\Gamma^\#_a\subseteq\Delta$ and by construction $f_a^\phi=g_a^\phi$ f.a. $\phi\in\mathcal{L}_{ELKy^r}$, thus $(\langle\Gamma,F,\vec{f}\rangle,\langle\Delta,G,\vec{g}\rangle)\in R_a^c$.
\end{proof}
\begin{lemma}[$K_a$ existence lemma]\label{lem:kaexistence}
For any $\langle\Gamma,F,\vec{f}\rangle\in W^c$, if $K_a\phi_1\not\in\Gamma$, there exists a $\langle\Delta,G,\vec{g}\rangle\in W^c$ with $(\langle\Gamma,F,\vec{f}\rangle,\langle\Delta,G,\vec{g}\rangle)\in R_a^c$ and $\neg\phi_1\in\Delta$.
\end{lemma}
\begin{proof}
Suppose $K_a\phi_1\not\in\Gamma$, i.e. $\neg K_a\phi_1\in\Gamma$ for some fixed $a\in\mathcal{A}$. For the desired properties, consider
\[
\Delta^-=\{\neg\phi_1\}\cup\{\phi\mid K_a\phi\in\Gamma\}
\]
To later extend this set to a full maximal consistent one, we establish the following:\\
\begin{claim}
$\Delta^-$ is consistent.
\end{claim}
\begin{claimproof}
Proof by contradiction, i.e. suppose that $\Delta^-$ is inconsistent. Then there exists some finite subset $\Theta=\{\psi_1,\dots,\psi_n\}\subseteq\Gamma^\#_a$ such that
\[
\vdash_\mathbb{SKYR}\bigwedge_{i=1}^n\psi_i\rightarrow\phi_1
\]
By \emph{(NK)} as well as distribution of $K_a$ over $\land$ and via the axiom \emph{(K)}, we have
\[
\vdash_\mathbb{SKYR}\bigwedge_{i=1}^nK_a\psi_i\rightarrow K_a\phi_1
\]
As $\psi_i\in\Gamma^\#_a$ f.a. $i\in\{1,\dots,n\}$, we have $K_a\psi_i\in\Gamma$ f.a. such $i$, i.e. $\bigwedge_{i=1}^nK_a\psi_i\in\Gamma$. We have, by deductive closure, that $K_a\phi_1\in\Gamma$. Contradiction.
\end{claimproof}
\newline
Let $\Delta$ be the extension of $\Delta^-$ to a max. consistent set(Prop. \ref{prop:lindenbaum}). Then, by Lem. \ref{lem:neighbourconst}, there exists a world $\langle\Delta,G,\vec{g}\rangle\in W^c$ s.t. $(\langle\Gamma,F,\vec{f}\rangle,\langle\Delta,G,\vec{g}\rangle)\in R_a^c$ with $\phi_1\not\in\Delta$ by construction.
\end{proof}
\begin{definition}
Let $t\in\hat{E}^c$. $s\in\hat{E}^c$ is a \emph{proper} subterm of $t$, denoted by $t\succ s$, if $t\neq s$ but $s$ occurs somewhere inside $t$.
\end{definition}
Note, that thus $s\not\succ s$. Note also that $\succ$ is obviously transitive, i.e. that $t\succ s$, $s\succ r$ implies $t\succ r$.
\begin{lemma}\label{lem:kyarexistencepre1}
Let $\langle\Gamma,F,\vec{f}\rangle\in W^c$ and $Ky_a^r(\phi_1,\phi_2)\not\in\Gamma$. For any $\phi_3\in\Gamma$ and for any $(s_{\phi_1},\phi_2)\in F$, there exists a world $\langle\Delta,G,\vec{g}\rangle\in W^c$ s.t. $\phi_3\in\Delta$, $(\langle\Gamma,F,\vec{f}\rangle,\langle\Delta,G,\vec{g}\rangle)\in R_a^c$ and $(s_{\phi_1},\phi_2)\not\in G$.
\end{lemma}
\begin{proof}
Let $a\in\mathcal{A}$ be fixed and $\langle\Gamma,F,\vec{f}\rangle\in W^c$ be as supposed, $\phi_3\in\Gamma$ and $(s_{\phi_1},\phi_2)\in F$. We construct $\langle\Delta,G,\vec{g}\rangle$ as follows:
\begin{enumerate}
\item $\Delta=\Gamma$
\item Take $\Psi=\{(t_\phi,\psi)\in F\mid Ky_a^r(\phi,\psi)\not\in\Gamma\}$ and $\Psi'=\{((t\cdot s)_\phi,\psi)\mid (t_\phi,\psi)\in\Psi\}$
\item $G_0= (F\setminus\Psi)\cup\Psi '$
\item $G_{n+1}=G_n\cup\{((t\cdot r)_{\alpha\land\beta},\psi)\mid (t_\alpha,\phi\rightarrow\psi),(r_\beta,\phi)\in G_n\}$ for all $n\geq 0$
\item $G=\bigcup_{n\in\mathbb{N}}G_n$
\item $g_b^\phi(\psi)=\begin{cases}f_b^\phi(\psi),&\text{if }(f_b^\phi(\psi),\psi)\not\in\Psi\\ (t\cdot s)_\phi,&\text{if }(f_b^\phi(\psi),\psi)\in\Psi\text{ and where }t_\phi=f_b^\phi(\psi)\end{cases}$
\end{enumerate}
\begin{claim}
$\langle\Delta,G,\vec{g}\rangle\in W^c$
\end{claim}
\begin{claimproof}
We check the conditions on $W^c$:
\begin{enumerate}
\item This follows again by the chain construction of $G$ and the built in closure conditions.
\item Let $\phi\in\Lambda$, i.e. $Ky_a^r(\top,\phi)\in\Gamma$ by \emph{(NKyR)} and $(e_\top,\phi)\in F$ as $\langle\Gamma,F,\vec{f}\rangle$ is well-defined, i.e. $(e_\top,\phi)\in F\setminus\Psi\subseteq G_0\subseteq G$.
\item Let $Ky_b^r(\phi,\psi)\land\phi\in\Gamma$ for some $b\in\mathcal{A}$. Thus $(f_b^\phi(\psi),\psi)\in F$. If $Ky_a^r(\phi,\psi)\in\Gamma$, then $(f_b^\phi(\psi),\psi)\not\in\Psi$ and thus $g_b^\phi(\psi)=f_b^\phi(\psi)$ and $(g_b^\phi(\psi),\psi)=(f_b^\phi(\psi),\psi)=(t_\phi,\psi)\in F\setminus\Psi\subseteq G_0\subseteq G$. If $Ky_a^r(\phi,\psi)\not\in\Gamma$, then $(f_b^\phi(\psi),\psi)\in\Psi$ and thus $(g_b^\phi(\psi),\psi)=((t\cdot s)_\phi,\psi)$ for $f_b^\phi(\psi)=t_\phi$ as well as $(g_b^\phi(\psi),\psi)\in\Psi'$ since $(f_b^\phi(\psi),\psi)\in\Psi$, i.e. $(g_b^\phi(\psi),\psi)\in\Psi'\subseteq G_0\subseteq G$.
\item As $f_b^\phi$ is well-defined, $g_b^\phi$ is well-defined as well.
\end{enumerate}
\end{claimproof}
Note, that if $Ky_a^r(\phi,\psi)\not\in\Gamma$, but $(t_\phi,\psi)\in G_0$, then $t\succ s$, as suppose $(t_\phi,\psi)\in G_0$, then either $(t_\phi,\psi)\in F\setminus\Psi$ or $(t_\phi,\psi)\in\Psi'$. For the former, $(t_\phi,\psi)\in F$ and $Ky_a^r(\phi,\psi)\in\Gamma$. Contradiction. For the latter, $t\succ s$ per definition.\newline
\begin{claim}
$(\langle\Gamma,F,\vec{f}\rangle,\langle\Delta,G,\vec{g}\rangle)\in R_a^c$
\end{claim}
\begin{claimproof}
As $\Delta=\Gamma$, obviously $\Gamma^\#_a\subseteq\Delta$. Also $\mathsf{dom}(g_a^\phi)=\mathsf{dom}(f_a^\phi)$ for any $\phi\in\mathcal{L}_{ELKy^r}$ as $\Gamma=\Delta$. Now, take any $f_a^\phi$ and $\psi\in\mathsf{dom}(f_a^\phi)$. As $\psi\in\mathsf{dom}(f_a^\phi)$, it is that $Ky_a^r(\phi,\psi)\in\Gamma$. If $(f_a^\phi(\psi),\psi)\not\in F$(as maybe $\phi\not\in\Gamma$), then $(f_a^\phi(\psi),\psi)\not\in\Psi$, as $\Psi\subseteq F$. Otherwise, still $(f_a^\phi(\psi),\psi)\not\in\Psi$ by $Ky_a^r(\phi,\psi)\in\Gamma$, i.e. either way $g_a^\phi(\psi)=f_a^\phi(\psi)$.
\end{claimproof}
\begin{claim}
If $Ky_a^r(\phi,\psi)\not\in\Gamma$ and $(t_\phi,\psi)\in G_{n+1}\setminus G_n$, then $t\succ s$.
\end{claim}
\begin{claimproof}
Let $Ky_a^r(\phi,\psi)\not\in\Gamma$. Induction on $n$:
\begin{description}[leftmargin=!,labelwidth=\widthof{\bfseries (IB):}]
\item [(IB)] Let $n=0$. Take $(t_\phi,\psi)\in G_1\setminus G_0$, i.e. $\exists (t'_{\phi'},\chi\rightarrow\psi),(t''_{\phi''},\chi)\in G_0$ s.t. $t_\phi=(t'_{\phi'}\cdot t''_{\phi''})$, i.e. $t=(t'\cdot t'')$ and $\phi=\phi'\land\phi''$. We distinguish three cases:
\begin{enumerate}[(i)]
\item $(t'_{\phi'},\chi\rightarrow\psi)\in\Psi'$: Then $t'\succ s$ per def., i.e. $t\succ s$.
\item $(t''_{\phi''},\chi)\in\Psi'$: Then $t''\succ s$ by def., i.e. $t\succ s$.
\item $(t'_{\phi'},\chi\rightarrow\psi),(t''_{\phi''},\chi)\not\in\Psi'$: Thus $(t'_{\phi'},\chi\rightarrow\psi),(t''_{\phi''},\chi)\in F\setminus\Psi$ and therefore $Ky_a^r(\phi',\chi\rightarrow\psi)\in\Gamma$ and $Ky_a^r(\phi'',\chi)\in\Gamma$. Thus by \emph{(EKyR)}, we have $Ky_a^r(\phi'\land\phi'',\psi)\in\Gamma$, i.e. $Ky_a^r(\phi,\psi)\in\Gamma$. Contradiction.
\end{enumerate}
\item [(IS)] Let $n>0$. Take $(t_\phi,\psi)\in G_{n+1}\setminus G_n$. Thus $\exists (t'_{\phi'},\chi\rightarrow\psi),(t''_{\phi''},\chi)\in G_n$ s.t. $t=(t'\cdot t'')$ and $\phi=\phi'\land\phi''$. We also have $Ky_a^r(\phi',\chi\rightarrow\psi)\not\in\Gamma$ or $Ky_a^r(\phi'',\chi)\not\in\Gamma$ as otherwise as above we have $Ky_a^r(\phi'\land\phi'',\psi)=Ky_a^r(\phi,\psi)\in\Gamma$ by \emph{(EKyR)}. Also $(t'_{\phi'},\chi\rightarrow\psi)\not\in G_{n-1}$ or $(t''_{\phi''},\chi)\not\in G_{n-1}$ as otherwise $(t_\phi,\psi)\in G_n$. We distinguish the following cases:
\begin{enumerate}[(i)]
\item $Ky_a^r(\phi',\chi\rightarrow\psi)\not\in\Gamma$, $(t'_{\phi'},\chi\rightarrow\psi)\in G_n\setminus G_{n-1}$: By \textbf{(IH)}, we have that $t'\succ s$, i.e. $t\succ s$.
\item $Ky_a^r(\phi',\chi\rightarrow\psi)\not\in\Gamma$, $(t'_{\phi'},\chi\rightarrow\psi)\in G_{n-1}$: We again divide in two cases:
\begin{enumerate}[(a)]
\item $(t'_{\phi'},\chi\rightarrow\psi)\in G_0$: By the above remark, we have $(t'_{\phi'},\chi\rightarrow\psi)\in G_0$ but $Ky_a^r(\phi',\chi\rightarrow\psi)\not\in\Gamma$, i.e. $t'\succ s$ and thus $t\succ s$.
\item $(t'_{\phi'},\chi\rightarrow\psi)\in G_{k+1}\setminus G_k$ for some $0\leq k\leq n-2$: We again apply \textbf{(IH)}, and derive $t'\succ s$, i.e. $t\succ s$.
\end{enumerate}
\item $Ky_a^r(\phi'',\chi)\not\in\Gamma$, $(t''_{\phi''},\chi)\in G_n\setminus G_{n-1}$: By \textbf{(IH)}, we have $t''\succ s$ and thus $t\succ s$.
\item $Ky_a^r(\phi'',\chi)\not\in\Gamma$, $(t''_{\phi''},\chi)\in G_{n-1}$: As above, we divide in two similar cases:
\begin{enumerate}[(a)]
\item $(t''_{\phi''},\chi)\in G_0$: By the one remark above, we have $(t''_{\phi''},\chi)\in G_0$ but $Ky_a^r(\phi'',\chi)\not\in\Gamma$, i.e. $t''\succ s$ and thus $t\succ s$.
\item $(t''_{\phi''},\chi)\in G_{k+1}\setminus G_k$ for some $0\leq k\leq n-2$: We again apply \textbf{(IH)}, and derive $t''\succ s$, i.e. $t\succ s$.
\end{enumerate}
\end{enumerate}
\end{description}
\end{claimproof}
\begin{claim}
$(s_{\phi_1},\phi_2)\not\in G$.
\end{claim}
\begin{claimproof}
We have $(s_{\phi_1},\phi_2)\in F$ but $Ky_a^r(\phi_1,\phi_2)\not\in\Gamma$. Suppose that $(s_{\phi_1},\phi_2)\in G$, i.e. either $(s_{\phi_1},\phi_2)\in G_0$, or $\exists k\in\mathbb{N}$, s.t. $(s_{\phi_1},\phi_2)\in G_{k+1}\setminus G_k$. For the former, we have that $s\succ s$. Contradiction. For the latter, by the previous claim, we have $s\succ s$. Contradiction. Thus $(s_{\phi_1},\phi_2)\not\in G$.
\end{claimproof}
\newline
Note, that $\phi_3\in\Delta$ as $\Delta=\Gamma$.
\end{proof}
\begin{lemma}\label{lem:kyarexistencepre2}
Let $\langle\Gamma,F,\vec{f}\rangle\in W^c$ and $Ky_a^r(\phi_1,\phi_2)\in\Gamma$ but $\phi_1\not\in\Gamma$. For each $\phi_3\in\Gamma$ and for any $(s_{\phi_1},\phi_2)\in F$, there exists a world $\langle\Delta,G,\vec{g}\rangle\in W^c$ s.t. $\phi_3\in\Delta$, $(\langle\Gamma,F,\vec{f}\rangle,\langle\Delta,G,\vec{g}\rangle)\in R_a^c$ and $(s_{\phi_1},\phi_2)\not\in G$.
\end{lemma}
\begin{proof}
Let $a\in\mathcal{A}$ be fixed and $\langle\Gamma,F,\vec{f}\rangle\in W^c$ with $Ky_a^r(\phi_1,\phi_2)\in\Gamma$ but $\phi_1\not\in\Gamma$. Take $\phi_3\in\Gamma$ and $(s_{\phi_1},\phi_2)\in F$. We again construct $\langle\Delta,G,\vec{g}\rangle$ as follows:
\begin{enumerate}
\item $\Delta=\Gamma$
\item Take $\Psi=\{(t_\phi,\psi)\in F\mid\phi\not\in\Gamma\}$ and $\Psi'=\{((t\cdot s)_\phi,\psi)\mid (t_\phi,\psi)\in\Psi\}$
\item $G_0= (F\setminus\Psi)\cup\Psi '$
\item $G_{n+1}=G_n\cup\{((t\cdot r)_{\alpha\land\beta},\psi)\mid (t_\alpha,\phi\rightarrow\psi),(r_\beta,\phi)\in G_n\}$ for all $n\geq 0$
\item $G=\bigcup_{n\in\mathbb{N}}G_n$
\item $g_b^\phi(\psi)=f_b^\phi(\psi)$ for all $b\in\mathcal{A}$, $\phi\in\mathcal{L}_{ELKy^r}$ and $\psi\in\mathsf{dom}(g_b^\phi)$
\end{enumerate}
\begin{claim}
$\langle\Delta,G,\vec{g}\rangle\in W^c$
\end{claim}
\begin{claimproof}
We check the conditions of $W^c$:
\begin{enumerate}
\item This follows directly by the chain construction of $G$.
\item Let $\phi\in\Lambda$, i.e. $(e_\top,\phi)\in F$, and as $\top\in\Gamma$ (Prop. \ref{prop:max_con_set_prop}), we have $(e_\top,\phi)\not\in\Psi$, i.e. $(e_\top,\phi)\in G_0\subseteq G$.
\item Let $Ky_b^r(\phi,\psi)\land\phi\in\Gamma$ for some $b\in\mathcal{A}$. Thus $(f_b^\phi(\psi),\psi)\in F$ and as $\phi\in\Gamma$ (Prop. \ref{prop:max_con_set_prop}) and $f_b^\phi(\psi)=t_\phi$, we have $(f_b^\phi(\psi),\psi)\not\in\Psi$, i.e. as $g_b^\phi(\psi)=f_b^\phi(\psi)$, we have $(f_b^\phi(\psi),\psi)=(g_b^\phi(\psi),\psi)\in G_0\subseteq G$.
\item As $\vec{f}$ is well-defined, $\vec{g}$ is well-defined as well.
\end{enumerate}
\end{claimproof}
Note, in analogy to the proof of Lemma \ref{lem:kyarexistencepre1}, that if $\phi\not\in\Gamma$ but $(t_\phi,\psi)\in G_0$, then $t\succ s$.
\begin{claim}
$(\langle\Gamma,F,\vec{f}\rangle,\langle\Delta,G,\vec{g}\rangle)\in R_a^c$
\end{claim}
\begin{claimproof}
As $\Delta=\Gamma$, we have $\Gamma^\#_a\subseteq\Delta$ and per definition, we have $g_a^\phi(\psi)=f_a^\phi(\psi)$.
\end{claimproof}
\begin{claim}
If $\phi\not\in\Gamma$ and $(t_\phi,\psi)\in G_{n+1}\setminus G_n$, then $t\succ s$.
\end{claim}
\begin{claimproof}
Let $\phi\not\in\Gamma$. Induction on $n$:
\begin{description}[leftmargin=!,labelwidth=\widthof{\bfseries (IB):}]
\item [(IB)] Let $n=0$. Let $(t_\phi,\psi)\in G_1\setminus G_0$, i.e. $\exists (t'_{\phi'},\chi\rightarrow\psi), (t''_{\phi''},\chi)\in G_0$ s.t. $t=(t'\cdot t'')$ and $\phi=\phi'\land\phi''$. We again divide in three cases:
\begin{enumerate}[(i)]
\item $(t'_{\phi'},\chi\rightarrow\psi)\in\Psi'$: By def., $t'\succ s$, i.e. $t\succ s$.
\item $(t''_{\phi''},\chi)\in\Psi'$: Again by def. of $\Psi'$, $t''\succ s$, i.e. $t\succ s$.
\item $(t'_{\phi'},\chi\rightarrow\psi), (t''_{\phi''},\chi)\not\in\Psi'$: As $(t'_{\phi'},\chi\rightarrow\psi)\not\in\Psi'$ we have $(t'_{\phi'},\chi\rightarrow\psi)\in F\setminus\Psi$, i.e. $(t'_{\phi'},\chi\rightarrow\psi)\in F$ and $(t'_{\phi'},\chi\rightarrow\psi)\not\in\Psi$, thus $\phi'\in\Gamma$ and similarly, as $(t''_{\phi''},\chi)\not\in\Psi'$, we have $\phi''\in\Gamma$. Thus $\phi=\phi'\land\phi''\in\Gamma$ by Prop. \ref{prop:max_con_set_prop}. Contradiction.
\end{enumerate}
\item [(IS)] Let $n>0$ and $(t_\phi,\psi)\in G_{n+1}\setminus G_n$, i.e. $\exists (t'_{\phi'},\chi\rightarrow\psi), (t''_{\phi''},\chi)\in G_n$ s.t. $t=(t'\cdot t'')$ and $\phi=\phi'\land\phi''$. Note that not both $(t'_{\phi'},\chi\rightarrow\psi), (t''_{\phi''},\chi)\in G_{n-1}$, as otherwise $(t_\phi,\psi)\in G_n$. Also not both $\phi',\phi''\in\Gamma$, as otherwise $\phi'\land\phi''=\phi\in\Gamma$ as before by Prop. \ref{prop:max_con_set_prop}. We divide between four cases:
\begin{enumerate}[(i)]
\item $\phi'\not\in\Gamma$, $(t'_{\phi'},\chi\rightarrow\psi)\in G_n\setminus G_{n-1}$: By \textbf{(IH)}, we have $t'\succ s$, i.e. $t\succ s$.
\item $\phi'\not\in\Gamma$, $(t'_{\phi'},\chi\rightarrow\psi)\in G_{n-1}$: We divide between the following cases:
\begin{enumerate}[(a)]
\item $(t'_{\phi'},\chi\rightarrow\psi)\in G_0$: Thus by the above remark, we have $t'\succ s$.
\item $(t'_{\phi'},\chi\rightarrow\psi)\in G_{k+1}\setminus G_k$ for some $0\leq k\leq n-2$: By \textbf{(IH)}, we derive $t'\succ s$, i.e. $t\succ s$.
\end{enumerate}
\item $\phi''\not\in\Gamma$, $(t''_{\phi''},\chi)\in G_n\setminus G_{n-1}$: By \textbf{(IH)}, we have $t''\succ s$, i.e. $t\succ s$.
\item $\phi''\not\in\Gamma$, $(t''_{\phi''},\chi)\in G_{n-1}$: We lastly divide between the following cases:
\begin{enumerate}[(a)]
\item $(t''_{\phi''},\chi)\in G_0$: Again, by the base case, we have $t''\succ s$.
\item $(t''_{\phi''},\chi)\in G_{k+1}\setminus G_k$ for some $0\leq k\leq n-2$: By \textbf{(IH)}, we derive $t''\succ s$, i.e. $t\succ s$.
\end{enumerate}
\end{enumerate}
\end{description}
\end{claimproof}
\begin{claim}
$(s_{\phi_1},\phi_2)\not\in G$.
\end{claim}
\begin{claimproof}
By supposition $(s_{\phi_1},\phi_2)\in F$, but $\phi_1\not\in\Gamma$. Suppose for contradiction, that $(s_{\phi_1},\phi_2)\in G$, i.e. $(s_{\phi_1},\phi_2)\in G_0$ or $(s_{\phi_1},\phi_2)\in G_{k+1}\setminus G_k$ for some $k\in\mathbb{N}$. For the former, we'd have $s\succ s$. Contradiction. For the latter, by the previously established claim, we have also $s\succ s$. Contradiction. Thus $(s_{\phi_1},\phi_2)\not\in G$.
\end{claimproof}
\newline
Again, we have that $\phi_3\in\Delta$ as $\Delta=\Gamma$.
\end{proof}
\begin{lemma}[$Ky_a^r$ existence lemma]\label{lem:kyarexistence}
Let $\langle\Gamma,F,\vec{f}\rangle\in W^c$ and $Ky_a^r(\phi_1,\phi_2)\not\in\Gamma$. For any $s_{\phi_3}\in E^c$, there exists a world $\langle\Delta,G,\vec{g}\rangle\in W^c$ s.t. $(\langle\Gamma,F,\vec{f}\rangle,\langle\Delta,G,\vec{g}\rangle)\in R_a^c$, $\phi_1\in\Delta$ and $(s_{\phi_3},\phi_2)\not\in G$.
\end{lemma}
\begin{proof}
Let $a\in\mathcal{A}$ be fixed and $\langle\Gamma,F,\vec{f}\rangle\in W^c$ s.t. $Ky_a^r(\phi_1,\phi_2)\not\in\Gamma$. Take $s_{\phi_3}\in E^c$. We divide between two cases: \emph{(i):} $Ky_a^r(\phi_3,\phi_2)\not\in\Gamma$ and \emph{(ii):} $Ky_a^r(\phi_3,\phi_2)\in\Gamma$.
\begin{enumerate}[(i)]
\item Let $Ky_a^r(\phi_3,\phi_2)\not\in\Gamma$. Suppose that $\phi_1\in\Gamma$. If $(s_{\phi_3},\phi_2)\not\in F$, there is nothing to do, as $\langle\Gamma,F,\vec{f}\rangle$ can reach itself. If $(s_{\phi_3},\phi_2)\in F$, then by Lem. \ref{lem:kyarexistencepre1}, we get the desired world.\\\\
Suppose that $\phi_1\not\in\Gamma$. Then $\Gamma^\#_a\cup\{\phi_1\}$ is still consistent, as supp. otherwise, then for some $\psi_1,\dots,\psi_n\in\Gamma^\#_a$:
\[
\vdash_\mathbb{SKYR}\bigwedge_{i=1}^n\psi_i\rightarrow\neg\phi_1
\]
i.e. by \emph{(NK)}, distribution over $\land$ and axiom \emph{(K)}, we have
\[
\vdash_\mathbb{SKYR}\bigwedge_{i=1}^nK_a\psi_i\rightarrow K_a\neg\phi_1
\]
Thus $K_a\neg\phi_1\in\Gamma$ and by \emph{(UKyR)}, we'd have $Ky_a^r(\phi_1,\phi_2)\in\Gamma$. Contradiction.\\\\
Thus, we extend this set $\Gamma^\#_a\cup\{\phi_1\}$ to a maximal consistent one, say $\Delta$. Now, there is a world $\langle\Delta,G,\vec{g}\rangle\in W^c$ with $(\langle\Gamma,F,\vec{f}\rangle,\langle\Delta,G,\vec{g}\rangle)\in R_a^c$ and $\phi_1\in\Delta$ by Lem. \ref{lem:neighbourconst}.

Also, as of Prop. \ref{prop:accesprop}, we have $Ky_a^r(\phi_1,\phi_2)\not\in\Delta$ and  $Ky_a^r(\phi_3,\phi_2)\not\in\Delta$. Again, if $(s_{\phi_3},\phi_2)\not\in G$, then there is nothing to do anymore. If $(s_{\phi_3},\phi_2)\in G$, then by Lem. \ref{lem:kyarexistencepre1}, we find a world $\langle\Theta,H,\vec{h}\rangle$ s.t. $(\langle\Delta,G,\vec{g}\rangle,\langle\Theta,H,\vec{h}\rangle)\in R_a^c$, i.e. $(\langle\Gamma,F,\vec{f}\rangle,\langle\Theta,H,\vec{h}\rangle)\in R_a^c$ by transitivity, and with $\phi_1\in\Theta$ and $(s_{\phi_3},\phi_2)\not\in H$.
\item Let $Ky_a^r(\phi_3,\phi_2)\in\Gamma$. Thus $\neg K_a(\phi_1\rightarrow\phi_3)\in\Gamma$, as otherwise $Ky_a^r(\phi_1,\phi_2)\in\Gamma$ by \emph{(IKyR)}. Thus, by Lem. \ref{lem:kaexistence}, there is a world $\langle\Delta,G,\vec{g}\rangle\in W^c$ s.t. $(\langle\Gamma,F,\vec{f}\rangle,\langle\Delta,G,\vec{g}\rangle)\in R_a^c$ and $\neg(\phi_1\rightarrow\phi_3)\in\Delta$, i.e. $\phi_1\land\neg\phi_3\in\Delta$ and thus $\phi_1\in\Delta$ and $\neg\phi_3\in\Delta$. Obviously, again by Prop. \ref{prop:accesprop}, we still have $Ky_a^r(\phi_1,\phi_2)\not\in\Delta$ and $Ky_a^r(\phi_3,\phi_2)\in\Delta$. If $(s_{\phi_3},\phi_2)\not\in G$, then we're done. Otherwise, by Lem. \ref{lem:kyarexistencepre2}, there exists a world $\langle\Theta,H,\vec{h}\rangle\in W^c$ s.t. $(\langle\Delta,G,\vec{g}\rangle,\langle\Theta,H,\vec{h}\rangle)\in R_a^c$, and thus $(\langle\Gamma,F,\vec{f}\rangle,\langle\Theta,H,\vec{h}\rangle)\in R_a^c$ again by transitivity, where we have $\phi_1\in\Theta$ and $(s_{\phi_3},\phi_2)\not\in H$.
\end{enumerate}
\end{proof}
Following from these considerations, we now propose the truth lemma corresponding to our canonical model $\mathfrak{M}^c$.
\begin{lemma}[Truth]\label{lem:truth}
$(\mathfrak{M}^c,\langle\Gamma,F,\vec{f}\rangle)\models\phi$ if and only if $\phi\in\Gamma$ for all $\phi\in\mathcal{L}_{ELKy^r}$ and all $\langle\Gamma,F,\vec{f}\rangle\in W^c$.
\end{lemma}
\begin{proof}
The proof is established by induction over the structure of the formula $\phi$.

For the \emph{induction base}, we consider\\
\textbf{(IB): }$\phi =p$: By the definition of $V^c$, $p\in\Gamma$ iff $\langle\Gamma,F,\vec{f}\rangle\in V^c(p)$ iff $(\mathfrak{M}^c, \langle\Gamma,F,\vec{f}\rangle)\models p$.

For the induction step we now distinguish the following cases for the different structures of the formula $\phi$:\\
\textbf{(IS): } \textbf{(i): }$\phi=\neg\psi$, \textbf{(ii): }$\phi=\psi\land\chi$, \textbf{(iii): }$\phi=K_a\psi$ and \textbf{(iv): }$\phi=Ky_a^r(\psi,\chi)$. As the first three cases are in some sense standard(s. for example Lemma 7.5, \cite{DHK2007}), we will just focus on item \textbf{(iv)}. In the following, let $w_\Gamma,w_\Delta,w_\Theta$ be shorthands for some structures $\langle\Gamma, F,\vec{f}\rangle, \langle\Delta, G,\vec{g}\rangle, \langle\Theta,H,\vec{h}\rangle\in W^c$.\\\\
First, suppose that $Ky_a^r(\psi,\chi)\in\Gamma$. Let $t_\psi:=f_a^\psi(\chi)\in E^c$. Consider an arbitrary world $w_\Delta$ with $(w_\Gamma,w_\Delta)\in R_a^c$ and additionally $(\mathfrak{M}^c,w_\Delta)\models\psi$. By \textbf{(IH)}, we have $\psi\in\Delta$ from the latter. As thus $Ky_a^r(\psi,\chi)\land\psi\in\Delta$ by Prop. \ref{prop:max_con_set_prop}, we have $(g_a^\psi(\chi),\chi)\in G$, i.e. $w_\Delta\in\mathcal{E}^c(g_a^\psi(\chi),\chi)$. As $(w_\Gamma,w_\Delta)\in R_a^c$, we have that $f_a^\psi(\chi)=g_a^\psi(\chi)=t_\psi$, i.e. $w_\Delta\in\mathcal{E}^c(t_\psi,\chi)$. We also have, as $Ky_a^r(\psi,\chi)\in\Delta$, that $K_a(\psi\rightarrow\chi)\in\Delta$ by \emph{(DKyR)} and thus that $\psi\rightarrow\chi\in\Delta$ by \emph{(T)}. As $\psi\in\Delta$, we have $\chi\in\Delta$ by \emph{(MP)}. By \textbf{(IH)}, we have $(\mathfrak{M}^c,w_\Delta)\models\chi$. Putting everything together, we have $(\mathfrak{M}^c,w_\Gamma)\models Ky_a^r(\psi,\chi)$.\\\\
On the other hand, suppose that $Ky_a^r(\psi,\chi)\not\in\Gamma$, therefore $\neg Ky_a^r(\psi,\chi)\in\Gamma$. We consider either \emph{(i):} $K_a(\psi\rightarrow\chi)\in\Gamma$ or \emph{(ii):} $K_a(\psi\rightarrow\chi)\not\in\Gamma$. 

For \emph{(i)}, by Lem. \ref{lem:kyarexistence}, we have that for any $s_{\psi'}\in E^c$, there exists a world $\langle\Delta,G,\vec{g}\rangle\in W^c$ s.t. $(\langle\Gamma,F,\vec{f}\rangle,\langle\Delta,G,\vec{g}\rangle)\in R_a^c$, $\psi\in\Delta$ (by \textbf{(IH)}, $(\mathfrak{M}^c,w_\Delta)\models\psi$) and $(s_{\psi'},\chi)\not\in G$. Thus there is \emph{no} single $s_{\psi'}\in E^c$ such that for all reachable worlds $w_\Delta$ with $\psi\in\Delta$ we have $(s_{\psi'},\chi)\in G$, i.e. there is \emph{no} single $s_{\psi'}\in E^c$ s.t. for all reachable worlds $w_\Delta$ with $(\mathfrak{M}^c,w_\Delta)\models\psi$, we have that $w_\Delta\in\mathcal{E}^c(s_{\psi'},\chi)$. By the semantics of $Ky_a^r$, we have $(\mathfrak{M}^c,w_\Gamma)\not\models Ky_a^r(\psi,\chi)$.

For \emph{(ii)}, i.e. $\neg K_a(\psi\rightarrow\chi)\in\Gamma$, by Lem. \ref{lem:kaexistence}, we have a world $w_\Delta$ with $(w_\Gamma,w_\Delta)\in R_a^c$ such that $\neg(\psi\rightarrow\chi)\in\Delta$. Therefore $\psi\land\neg\chi\in\Delta$ and by \textbf{(IH)}, we have $(\mathfrak{M}^c,w_\Delta)\models\psi$ and $(\mathfrak{M}^c,w_\Delta)\not\models\chi$. Thus $(\mathfrak{M}^c,w_\Gamma)\not\models Ky_a^r(\psi,\chi)$.
\end{proof}
\begin{theorem}[Completeness of $\mathbb{SKYR}$ over $\mathcal{K}y^r\mathcal{S}5$]
$\Gamma\models_{\mathcal{K}y^r\mathcal{S}5}\phi$ implies $\Gamma\vdash_{\mathbb{SKYR}}\phi$.
\end{theorem}
\begin{proof}
Suppose $\Gamma\not\vdash_\mathbb{SKYR}\phi$, then $\Gamma\cup\{\neg\phi\}$ is obviously consistent. Now, let $\Gamma'$ be the extensions of $\Gamma\cup\{\neg\phi\}$ to a maximal consistent set. By the definition of the canonical model and Prop. \ref{prop:worldexistence}, there exists at least one world for this set in $W^c$, namely $\langle\Gamma',F^{\Gamma'},\vec{f}^{\Gamma'}\rangle$. By Lem. \ref{lem:truth}, we have $(\mathfrak{M}^c,\langle\Gamma',F^{\Gamma'},\vec{f}^{\Gamma'}\rangle)\models\Gamma\cup\{\neg\phi\}$ as $\Gamma\cup\{\neg\phi\}\subseteq\Gamma'$, and thus $\Gamma\cup\{\neg\phi\}$ is satisfiable. Therefore, we have $\Gamma\not\models_{\mathcal{K}y^r\mathcal{S}5}\phi$.
\end{proof}
\section{Expressivity comparisons}
As it was said at the beginning of the paper, the initial motivation was to study the dynamic extensions for the logic of knowing-why, with a first look at public announcements. As the logic $\mathbf{ELKy^r}$ was presented as some sort of workaround for problems concerning the logic $\mathbf{PAFKy}$(or more precise, concerning the classic style of axiomatization for "public announcement"-type logics), we now make expressivity comparisons between the newly introduced logics of this paper.\\

Before proceeding, it may additionally be notable that we find that $Ky_a^r(\phi,\psi)$, although sugested in the final section of \cite{XWS2016}, does not directly correspond to $[\phi]Ky_a\psi$. For this, one may easily imagine the following model $\mathfrak{M}$:
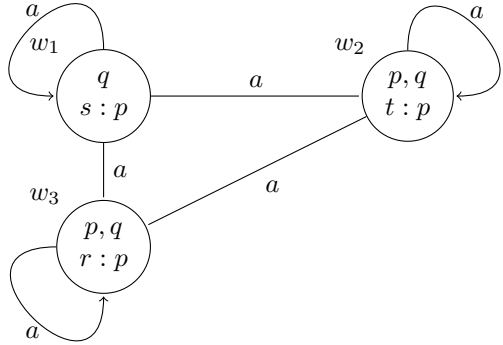
\begin{center}
\begin{tikzpicture}[shorten >=1pt,node distance=2cm and 4cm,on grid,auto]
\node[state, label=above left:$w_1$, align=center] (w_1) {$q$ \\ $s:p$};
\node[state, label=above left:$w_2$, align=center] (w_2) [right=of w_1] {$p,q$ \\ $t:p$};
\node[state, label=above left:$w_3$, align=center] (w_3) [below=of w_1] {$p,q$ \\ $r:p$};
\path[-] (w_1) edge node {$a$} (w_2)
               edge node {$a$} (w_3);
\path[-] (w_2) edge node {$a$} (w_3);
\draw[->,shorten >=1pt] (w_1) to [out=90,in=180,loop,looseness=4.8] node[above] {$a$} (w_1);
\draw[->,shorten >=1pt] (w_2) to [out=90,in=0,loop,looseness=4.8] node[above] {$a$} (w_2);
\draw[->,shorten >=1pt] (w_3) to [out=180,in=270,loop,looseness=4.8] node[below] {$a$} (w_3);
\end{tikzpicture}
\end{center}
Since $(\mathfrak{M},w_1)\not\models p$, we have that $(\mathfrak{M},w_1)\models [p]Ky_a q$ from the semantics of $[\cdot]$. At the same time, we find that $(\mathfrak{M},w_1)\not\models Ky_a^r(p,q)$ as, although $w_1$ is not considered in the evaluation of the $\mathcal{E}$-clause because of the before mentioned condition, we still find that there does not exists a uniform $t\in E$ for the left-to-consider worlds $w_2$ and $w_3$.\\

The main difference exploited here is the missing $\phi$-implication in the semantical definition of $Ky_a^r(\phi,\psi)$. In the following argument though, it can also be seen that even a corresponding modification has no possibility in providing an adequate translation. For this, we introduce a second concept from the original paper \cite{XWS2016}.
\begin{definition}[Factivity Property]
A model $\mathfrak{M}$ has the factivity property(is factive), if whenever $w\in\mathcal{E}(t,\phi)$, then $(\mathfrak{M},w)\models\phi$.
\end{definition}
Given a model $\mathfrak{M}=\langle W, E, \{R_a\mid a\in\mathcal{A}\},\mathcal{E}, V\rangle$, one may construct its factive companion $\mathfrak{M}^F=\langle W, E, \{R_a\mid a\in\mathcal{A}\},\mathcal{E}^F, V\rangle$ where
\[
\mathcal{E}^F(t,\phi)=\mathcal{E}(t,\phi)\setminus\{w\in W\mid (\mathfrak{M},w)\not\models\phi\}
\]
Obviously, for a factive model $\mathfrak{M}$, $\mathfrak{M}$ and $\mathfrak{M}^F$ coincide. The following lemma now asserts that the $\mathbf{ELKy}$-formulas are neutral in respect to facitvity.
\begin{lemma}[Xu, Wang, Studer \cite{XWS2016}]\label{lem:elkyfactivity}
For any $\phi\in\mathcal{L}_{ELKy}$, any $\mathbf{ELKy}$-model $\mathfrak{M}$ and any $w\in\mathcal{D}(\mathfrak{M})$, $(\mathfrak{M},w)\models\phi$ if and only if $(\mathfrak{M}^F,w)\models\phi$.
\end{lemma}
We obtain the following generalization for $\mathbf{ELKy^r}$.
\begin{lemma}\label{lem:elkyrfactivity}
For any $\phi\in\mathcal{L}_{ELKy^r}$, any $\mathbf{ELKy^r}$-model $\mathfrak{M}$ and any $w\in\mathcal{D}(\mathfrak{M})$, $(\mathfrak{M},w)\models\phi$ if and only if $(\mathfrak{M}^F,w)\models\phi$.
\end{lemma}
\begin{proof}
Proof by induction on the structure of formulas. We leave the classical propositional and modal cases unconsidered, as $\mathfrak{M}$ is only possibly different from $\mathfrak{M}^F$ in the explanation function. Thus consider $Ky_a^r(\phi,\psi)$:\\

Suppose $(\mathfrak{M},w)\models Ky_a^r(\phi,\psi)$, i.e. $\exists t\in E:\forall v\in W:(w,v)\in R_a$ and $(\mathfrak{M},v)\models\phi$ implies $(\mathfrak{M},v)\models\psi$ and $v\in\mathcal{E}(t,\psi)$. Thus, for all those $v$, $v\not\in\{u\mid (\mathfrak{M},u)\not\models\psi\}$, i.e. $v\in\mathcal{E}^F(t,\psi)$ as $(\mathfrak{M},v)\models\psi$. By (IH), we have $(\mathfrak{M}^F,w)\models Ky_a^r(\phi,\psi)$.\\

Suppose otherwise that $(\mathfrak{M}^F,w)\models Ky_a^r(\phi,\psi)$, i.e. $\exists t\in E:\forall v\in W:(w,v)\in R_a$ and $(\mathfrak{M}^F,v)\models\phi$ implies $(\mathfrak{M}^F,v)\models\psi$ and $v\in\mathcal{E}^F(t,\psi)$. By $v\in\mathcal{E}^F(t,\psi)$, we automatically have $v\in\mathcal{E}(t,\psi)$ and by (IH), we obtain $(\mathfrak{M},w)\models Ky_a^r(\phi,\psi)$.
\end{proof}
\begin{lemma}\label{lem:pafkynotfactive}
There exists a formula $\phi\in\mathcal{L}_{PAFKy}$, a model $\mathfrak{M}$ and $w\in\mathcal{D}(\mathfrak{M})$ such that $(\mathfrak{M},w)\models\phi$ but $(\mathfrak{M}^F,w)\not\models\phi$.
\end{lemma}
\begin{proof}
Consider the following model $\mathfrak{M}$(reflexive arrows are not shown, but expected)
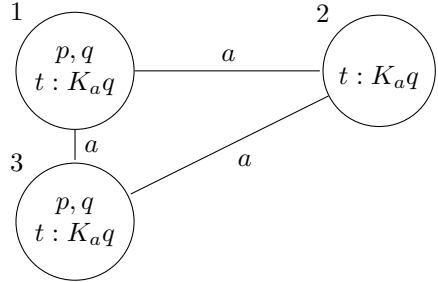
\begin{center}
\begin{tikzpicture}[shorten >=1pt,node distance=2cm and 4cm,on grid,auto]
\node[state, label=above left:$1$, align=center] (w_1) {$p,q$ \\ $t:K_aq$};
\node[state, label=above left:$2$, align=center] (w_2) [right=of w_1] {\\ $t:K_aq$};
\node[state, label=above left:$3$, align=center] (w_3) [below=of w_1] {$p,q$ \\ $t:K_aq$};
\path[-] (w_1) edge node {$a$} (w_2)
               edge node {$a$} (w_3)
         (w_2) edge node {$a$} (w_3);
\end{tikzpicture}
\end{center}
and its factive counterpart $\mathfrak{M}^F$, where we loose all explanations for $K_aq$ as in every world $x$, by $(\mathfrak{M},2)\not\models q$, we find that $(\mathfrak{M},x)\not\models K_aq$ as the relation $R_a$ is total among these worlds:
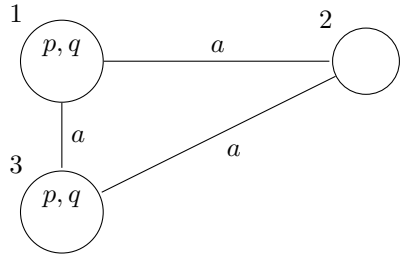
\begin{center}
\begin{tikzpicture}[shorten >=1pt,node distance=2cm and 4cm,on grid,auto]
\node[state, label=above left:$1$, align=center] (w_1) {$p,q$ \\};
\node[state, label=above left:$2$, align=center] (w_2) [right=of w_1] {\\};
\node[state, label=above left:$3$, align=center] (w_3) [below=of w_1] {$p,q$ \\};
\path[-] (w_1) edge node {$a$} (w_2)
               edge node {$a$} (w_3)
         (w_2) edge node {$a$} (w_3);
\end{tikzpicture}
\end{center}
Consider the formula $\phi:=[p]Ky_aK_aq$. We have $(\mathfrak{M},1)\models\phi$. For this, first note that $(\mathfrak{M},1)\models p$. Second, $\forall v\in W:(1,v)\in R_a$ and $(\mathfrak{M},v)\models p$ leaves us with worlds $1,3$. For those, we obtain both instances of $(\mathfrak{M}|p,v)\models K_a q$ as $[[p]]_\mathfrak{M}=\{1,3\}$ and $(\mathfrak{M},v)\models q$. At last, note that $1,3\in\mathcal{E}(t,K_aq)$.\\
At the same time, we have $(\mathfrak{M}^F,1)\not\models\phi$ since, although similarly $[[p]]_\mathfrak{M}=\{1,3\}$, there exits no $t\in E$ such that $1,3\in\mathcal{E}^F(t,K_aq)$ as $\mathcal{E}^F(t,K_aq)$ has to be empty for every $t$.
\end{proof}
\begin{theorem}\label{thm:pafkyelkyrexpressivity}
$\mathbf{PAFKy}$ and $\mathbf{ELKy^r}$ are not equally expressive.
\end{theorem}
\begin{proof}
Suppose that $\mathbf{PAFKy}$ and $\mathbf{ELKy^r}$ are equally expressive. Thus, there exists a translation function $t$ from formulas of $\mathbf{PAFKy}$ to formulas of $\mathbf{ELKy^r}$ such that $\phi\equiv\phi^t$ for any $\phi\in\mathcal{L}_{PAFKy}$. Then the following diagram
\begin{center}
\begin{tikzpicture}
\node at (0,1.5) {$(\mathfrak{M},w)\models\phi$};
\node at (0,0) {$(\mathfrak{M},w)\models\phi^t$};
\node at (3.35,0) {$(\mathfrak{M}^F,w)\models\phi^t$};
\node at (3.3,1.5) {$(\mathfrak{M}^F,w)\models\phi$};
\node at (0.2,0.75) {$t$};
\node at (3.55,0.75) {$t$};
\node at (1.6,0.3) {Lem. \ref{lem:elkyrfactivity}};
\draw[implies-implies,double equal sign distance] (0,1.25) -- (0,0.25);
\draw[implies-implies,double equal sign distance] (1,0) -- (2.2,0);
\draw[implies-implies,double equal sign distance] (3.35,0.25) -- (3.35,1.25);
\end{tikzpicture}
\end{center}
may be completed at the top for any $\phi\in\mathcal{L}_{PAFKy}$, any model $\mathfrak{M}$ and any $w\in\mathcal{D}(\mathfrak{M})$, which is a contradiction to Lem. \ref{lem:pafkynotfactive}.
\end{proof}
\begin{corollary}
$\mathbf{PAFKy^r}$ has greater expressivity than $\mathbf{ELKy^r}$.\footnote{Note, that in a similar way one may obtain a different proof for Thm. \ref{thm:pafkygreaterelky}, using Lem. \ref{lem:elkyfactivity}.}
\end{corollary}
\begin{proof}
Since $\mathbf{ELKy^r}$ is contained in $\mathbf{PAFKy^r}$, it can't be more expressive. By Thm. \ref{thm:pafkyelkyrexpressivity} and Lem. \ref{lem:pafkynotfactive}, there exists a formula $\phi\in\mathcal{L}_{PAFKy}$ which can't be expressed in $\mathbf{ELKy^r}$. As $\mathbf{PAFKy^r}$ contains $\mathbf{PAFKy}$, via the association $Ky_a\phi\sim Ky_a^r(\top,\phi)$, we have greater expressivity.
\end{proof}
By this argument it can be seen that a relativized knowing-why operator does not suffice to imitate public announcement behavior in the logic of knowing why, thus leaving a rest of void it initially intended to fill. Although we leave further inspections for future work, we still want to advocate for the consideration of the use of non-standard semantics for public announcements, namely context-dependent semantics, which seems like a promising alternative way of dealing with the laid out problems.
\subsection{Using non-standard semantics}
The following semantic concepts for public announcement logic are mainly due to Wang in their current form, see \cite{WC2013} and \cite{Wan2006}. For providing a small recollection of the basic notions of context-dependent semantics of the operator $[\phi]$, we consider another semantic relation $\models_\rho$, similar to the classical $\models$, but induced with a formula $\rho\in\mathcal{L}_{PAFKy^r}$ providing a specific evaluation context. For the behavior of this new relation, one may consider the following(only the interesting cases are provided, i.e. the out-carrying of the relation over $\land,\neg$ is obviously following the old structure as the main difference lies in the many-world context).
\begin{align*}
&(\mathfrak{M},w)\models\phi\Leftrightarrow (\mathfrak{M},w)\models_\top\phi\\
&(\mathfrak{M},w)\models_\rho p\text{ iff }w\in V(p)\\
&\dots\\
&(\mathfrak{M},w)\models_\rho K_a\phi\text{ iff }\forall v\in W:(w,v)\in R_a\text{ and }(\mathfrak{M},v)\models_\top\rho\text{ implies }(\mathfrak{M},v)\models_\rho\phi\\
&(\mathfrak{M},w)\models_\rho [\psi]\phi\text{ iff }(\mathfrak{M},w)\models_\top\text{ implies }(\mathfrak{M},w)\models_{\rho\land\psi}\phi
\end{align*}
As it was shown in \cite{WC2013}, the composition axiom in this semantics turns out to be much simpler with
\begin{equation*}
[\phi][\psi]\chi\leftrightarrow [\phi\land\psi]\chi
\end{equation*}
This new induced handling of compositions of public announcement should make it possible to find a simple reduction axiom like $\phi\rightarrow Ky_a^r(\phi\land\psi,\chi)$ as a proposition. But on deeper insights, the whole idea of a conditionalized versions of the knowing-why operator may(or should) be completely unnecessary in this new context. The axiomatization and the more intensive study of this semantics in this context are left as future work.\\

Another use of different semantics may lie in restructuring the semantics of $Ky_a^r$ itself. Such a different semantic definition shall obviously provide a possibility for translation between the relativized versions and the versions incorporating the public announcements. This, for example, is enabled by the following semantics:
\[
(\mathfrak{M},w)\models Ky_a^r(\phi,\psi)\text{ iff }\exists t\in E:\forall v\in W: (w,v)\in R_a\text{ and }(\mathfrak{M},v)\models\phi\text{ implies }(\mathfrak{M}|\phi,w)\models\psi\text{ and }v\in\mathcal{E}(t,\psi)
\]
Obviously, through the enforcement of the $\phi$-restricted model \emph{only} in the continuing evaluation of $\phi$, we avoid the problems laid out before. From this, we may translate $[\phi]Ky_a\psi$ to $\phi\rightarrow Ky_a^r(\phi,\psi)$, although it seems to require a quite different approach of axiomatization. Additionally, the usual idea of introducing a relativized operator to \emph{leave} the context updated models, is discarded which doesn't get along with the usual spirit.
\section{Conclusions}
In this paper, the logic of knowing why under the extension with public announcement operators for formulas was addressed. Through the difficulties arising with providing of a reduction-based axiomatic system, we considered another logical operator, namely the conditionalized version of the basic $Ky$-operator to provide a partial workaround. Following to this, as the main result of this paper, we proved the newly introduced axiomatic system concerning this logic using the relativized operator as being sound and complete with respect to the basic $\mathcal{S}5$-class of models following the definition of the initial paper \cite{XWS2016}. A conditionalized version of the canonical model presented in \cite{XWS2016} and \cite{WF2013},\cite{WF2014} is here provided in order to achieve these results. In the following section, we then considered the problem of expressivity between the different logics in discourse, where we found that $\mathbf{ELKy^r}$ does not fulfill its reduction promise as hoped and so lies as an intermediate among $\mathbf{ELKy}$, $\mathbf{PAFKy}$ and $\mathbf{PAFKy^r}$ concerning its expressive power.\\

The situation was at first sight similar to the problem of a reduction style axiomatization of common knowledge with public announcement operators(or similar dynamic notions), as examined in \cite{BMS1998}, \cite{BEK2006}. The main difference and cause for problems to those approaches is the sensitivity of the argument of the knowing why operator, as the core syntactical structure of a formula is needed for the evaluation of the $\mathcal{E}$-function and we thus can't use similar formulas, in the sense of being equal under satisfaction. These problems did not arise with previous attempts of adding public announcements to other non-standard epistemic logics, e.g. \cite{WF2013}, \cite{WF2014}, \cite{EGW2016}, as the non-classical operators in these contexts concerned objects disconnected from the set of well-formed formulas and the other classical Kripkean parts of the model definition.
\bibliographystyle{plain}
\bibliography{ref}

\end{document}